\documentclass[a4paper,10pt, oneside]{article}
\usepackage{amsmath, amssymb, amsthm,graphicx}
\usepackage[font=small,labelfont=md,textfont=it]{caption}
\usepackage{floatrow,float}
\usepackage[titletoc, title]{appendix}
\usepackage[colorlinks,linkcolor=blue,citecolor=blue]{hyperref}
\usepackage{longtable}
\usepackage{pgfplots}
\usepackage{diagbox}
\usepackage{booktabs,makecell,multirow}
\usepackage[capitalise, nosort]{cleveref}
\usepackage{cases,color}
\crefname{equation}{}{}
\crefname{lemma}{Lemma}{Lemmas}
\crefname{thm}{Theorem}{Theorems}
\usepackage{algorithm}
\usepackage{algorithmic}
\usepackage{amsmath,bm}

\newcommand{\inner}[1]{\left( {#1} \right)}
\newcommand{\dual}[1]{\left\langle {#1} \right\rangle}
\newcommand{\prox}[0]{ {\bf prox}}
\newcommand{\dom}[0]{ {\bf dom\,}}
\newcommand{\argmin}[0]{ {\rm argmin}}

\newcommand{\nm}[1]{\left\lVert {#1} \right\rVert}
\newcommand{\snm}[1]{\left\lvert {#1} \right\rvert}
\newcommand{\ssnm}[1]
{
  \left\vert\kern-0.25ex
  \left\vert\kern-0.25ex
  \left\vert
  {#1}
  \right\vert\kern-0.25ex
  \right\vert\kern-0.25ex
  \right\vert
}

\makeatletter
\def\spher@harm#1{%
  \vbox{\hbox{%
    \offinterlineskip
    \valign{&\hb@xt@2\p@{\hss$##$\hss}\vskip.2ex\cr#1\crcr}%
  }\vskip-.36ex}%
}
\def\gshone{\spher@harm{.}}
\def\gshtwo{\spher@harm{.&.}}
\def\gshthree{\spher@harm{.&.&.}}
\let\gsh\spher@harm
\makeatother

\newtheorem{Def}{Definition}[section]

\newtheorem{lemma}{Lemma}[section]
\newtheorem{remark}{Remark}[section]
\newtheorem{theorem}{Theorem}[section]
\newtheorem{proposition}{Proposition}[section]

\makeatletter\def\@captype{table}\makeatother

\begin{document}

	\title{From differential equation solvers to
		accelerated first-order methods for convex optimization
		\thanks{Hao Luo was supported by the China Scholarship Council (CSC) joint Ph.D. student scholarship (Grant 201806240132).}
	}
	
	
	%
	%
	\author{	
			Hao Luo \thanks{Email: galeolev@foxmail.com},
	Long Chen \thanks{Corresponding author. Email: chenlong@math.uci.edu}. \\
}
	\date{}

	\maketitle
	
	\begin{abstract}
	Convergence analysis of accelerated first-order methods 
for convex optimization problems 
are presented from the point of view of ordinary differential 
equation solvers. A new dynamical system, called Nesterov accelerated gradient flow, has been
derived from the connection between acceleration mechanism and $A$-stability of ODE solvers, and the exponential decay of a tailored Lyapunov function along with the solution trajectory is proved. Numerical discretizations are then considered and convergence rates are 
established via a unified discrete Lyapunov function. 
The proposed differential equation solver approach can 
not only cover existing accelerated methods, 
such as FISTA, G\"{u}ler's proximal algorithm and Nesterov's accelerated gradient method, but also produce  
new algorithms for composite convex optimization that possess accelerated convergence rates. 
	\end{abstract}
	
	
	
\section{Introduction}
We consider iterative methods for solving the unconstrained minimization problem
\begin{equation}\label{min}
	\min_{x\in V} f(x),
\end{equation}
where $V$ is a Hilbert space, and 
$f:V\to\mathbb R\cup\{+\infty\}$ is a properly closed convex function. 
We shall first consider smooth $f$ on the entire space $V$ and later focus on the composite case $f = h + g$ where both $h$ (smooth) and $g$ (non-smooth)
are convex on some (simple) closed convex set $Q\subseteq V$. 
We are mainly interested in the development and analysis of accelerated
first-order methods.

Suppose $V$ is equipped with the inner product $(\cdot,\cdot)$ and the correspondingly induced norm $\nm{\cdot}$.
We use 
$\dual{\cdot,\cdot}$ to denote the duality pair 
between $V^*$ and $V$, 
where $V^*$ is the continuous dual space of $V$ and is endowed with the conventional dual norm $\nm{\cdot}_{*}$. For any interval 
$I\subseteq \mathbb R$,
denote by $C^k(I;V)$ the space of all $k$-times 
continuous differentiable $V$-valued functions on $I$, 
and the superscript $k$ is dropped when $k=0$. 
Let $\Omega\subseteq V$ be some closed convex subset, we 
say $f\in\mathcal S_{\mu}^{1}(\Omega)$ if it is 
continuous differentiable on $\Omega$ and there exists 
$\mu\geqslant 0$ such that
\begin{equation}\label{eq:convex-mu}
	f(x) - f(y) - \langle \nabla f(y), x - 
	y \rangle  \geqslant \frac{\mu}{2} \| x- y \|^2
	\quad \forall\, x,y \in \Omega.
\end{equation}
We call~\eqref{eq:convex-mu} the $\mu$-convexity of $f$ and when $\mu>0$, we say $f$ is strongly convex. We also 
write $f\in\mathcal S^{1,1}_{\mu,L}(\Omega)$ if 
$f\in\mathcal S_{\mu}^{1}(\varOmega)$ and $\nabla f$ is Lipschitz 
continuous on $\Omega$: there exists $0<L<\infty$ such that
\begin{equation}\label{eq:Lip-L}
	\|\nabla f(x) - \nabla f(y)\|_{*} 
	\leqslant L\| x - y \| \quad \forall\, x,y \in \Omega.
\end{equation}
By~\cite[Theorem 2.1.5]{Nesterov:2013Introductory},
this implies the inequality
\begin{equation}\label{eq:Lip-L-equiv}
	f(x) - f(y) - \langle \nabla f(y), x - 
	y \rangle  \leqslant \frac{L}{2} \| x- y \|^2
	\quad \forall\, x,y \in \Omega.
\end{equation}
For $\Omega = V$, we shall write $\mathcal S_{\mu}^{1}(\Omega)$ 
and $\mathcal S^{1,1}_{\mu,L}(\Omega)$ as $\mathcal S_{\mu}^{1}$ and 
$\mathcal S^{1,1}_{\mu,L}$, respectively. 

The above functional classes are what we work with in this paper. As for the optimization problem \eqref{min}, we also care about the global minimizer(s) of $f$. For strongly convex $f\in\mathcal S_{\mu}^1(Q)$ with $\mu>0$, it is well-known that the minimizer exists uniquely. However, for convex case $\mu=0$, to promise the existence of minimizers, additional assumption, such as coercivity condition, which means $f(x)\to\infty$ when $\nm{x}\to\infty$, is usually imposed. Throughout, we denote by $\argmin f$ the set of global minimizers of \eqref{min} and assume it is nonempty.

One approach to derive the gradient descent (GD)
method is discretizing an ordinary differential 
equation (ODE), i.e., the so-called gradient flow:
\begin{equation}\label{gradientflow}
	x'(t) = - \nabla f(x(t)),\quad t>0.
\end{equation}
Here we introduce an artificial time variable $t$ and $x'$ is 
the derivative taken with respect to $t$. 
For ease of notation, in the sequel, we shall omit $t$ 
when no confusion arises. The simplest 
forward (explicit) Euler method with step size $\eta_k>0$ 
leads to the GD method 
\begin{equation*}
	x_{k+1} = x_k - \eta_k \nabla f(x_k).    
\end{equation*}
In the terminology of numerical analysis, it is well-known that 
this method is {\it conditionally $A$-stable} (cf. Section \ref{sec:stab}), 
and for $f\in\mathcal S_{\mu,L}^{1,1}$ with $0\leqslant \mu\leqslant L<\infty$, the 
step size $\eta_k=1/L$ is allowed to get the rate (see \cite[Chapter 2]{Nesterov:2013Introductory})
\begin{equation}\label{eq:rate-gd}
	O\left(\min\big\{L/k,(1+\mu/L)^{-k}\big\}\right).
\end{equation}
One can also consider the backward (implicit) Euler method 
\begin{equation}\label{eq:intro-implicit}
	x_{k+1} = x_k - \eta_k \nabla f(x_{k+1}),
\end{equation}
which is {\it unconditionally $A$-stable} (cf. Section \ref{sec:stab}) and 
coincides with the well-known proximal point algorithm (PPA)~\cite{Parikh:2014}
\begin{equation}
	\label{eq:prox}
	x_{k+1} = \prox_{\eta_k f}(x_k):=\mathop{\argmin}\limits_{y\in V} 
	\left ( f(y) + \frac{1}{2\eta_k}\| y - x_k \|^2 \right ).
\end{equation}
Note  that this method allows $f$ to be nonsmooth and possesses linear 
convergence rate even for convex objective, as long as 
$\eta_k\geqslant\eta>0$ for all $k>0$.
\subsection{Main results}	
Let us start from the quadratic objective $f(x) = \frac{1}{2}x^{\top}Ax$ over ${\mathbb R}^d$, for which the gradient flow~\eqref{gradientflow} reads simply as 
\begin{equation}\label{eq:x'Ax}
	x' = -Ax,
\end{equation}
where $A$ is symmetric positive semi-definite and makes $f\in\mathcal S_{\mu,L}^{1,1}$.
Instead of solving ~\eqref{eq:x'Ax}, we turn to a general linear ODE system
\begin{equation}\label{eq:yGy-intro}
	y'=Gy.
\end{equation}
Briefly speaking, our main idea is to seek such a system~\eqref{eq:yGy-intro} with some asymmetric block matrix $G$ that transforms the spectrum of $A$ from the real line to the complex plane and reduces the condition number from $\kappa(A) = L/\mu$ to $\kappa(G) = O(\sqrt{L/\mu})$. Afterwards, accelerated gradient methods may be constructed from $A$-stable methods for solving~\eqref{eq:yGy-intro}  with a significant larger step size and consequently improve the contraction rate from $O((1-\mu/L)^k)$ to $O((1-\sqrt{\mu/L})^k)$.  
Furthermore, to handle the convex case $\mu=0$, we combine the transformation $G$ with suitable time scaling technique; for more details, we refer to Section \ref{sec:stab}.

One successful and important transformation example is given below
\begin{equation}\label{eq:yGy}
	G = 
	\begin{pmatrix}
		- I & \quad I \\
		\mu/\gamma -A/\gamma & \quad -\mu/\gamma \, I
	\end{pmatrix},
\end{equation}
where 
the built-in scaling factor $\gamma$ 
is positive and satisfies 
\begin{equation}\label{eq:gama}
	\gamma' =\mu-\gamma,\quad\gamma(0)=\gamma_0>0.
\end{equation}
Based on this, for general $f\in\mathcal S_\mu^1$ with $\mu\geqslant 0$, we replace $A$ in~\eqref{eq:yGy} with $\nabla f$ and write $y=(x,v)$ to obtain a first-order dynamical system:
\begin{equation}\label{eq:intro-NAG-system}
	\left\{
	\begin{split}
		x' = {}&v-x,\\
		v'={}&\frac{\mu}{\gamma}(x-v) - \frac{1}{\gamma}\nabla f(x),
	\end{split}
	\right.
\end{equation}
where $\gamma$ solves \eqref{eq:gama}.
Eliminating $v$, we arrive at a second-order ODE of $x$:
\begin{equation}\label{eq:ode-NAG-intro}
	\gamma x''+
	\left(\mu+\gamma\right)x'
	+\nabla f(x)=0,
\end{equation}
which is actually a heavy ball model
(cf.~\eqref{eq:heavy ball-ode}) with novel variable damping coefficients in front of $x''$ and $x'$.
Thanks to the scaling factor $\gamma$, we can handle both the convex case ($\mu = 0$) and the strongly convex case ($\mu > 0$) in a unified way. 	Moreover, we shall prove that for $\mu\geqslant 0$, there holds the exponential decay property
\begin{equation}\label{eq:conv-Lt-intro}
	\mathcal L(t)\leqslant e^{-t}\mathcal L(0),\quad t>0,
\end{equation}
for a tailored Lyapunov function 
\begin{equation}\label{eq:Lt-intro}
	\mathcal L(t)=
	f(x(t))-f(x^*)+\frac{\gamma(t)}{2}
	\nm{v(t)-x^*}^2,\quad t>0,
\end{equation}
where $x^*\in\argmin f$ is a global minimizer of $f$.

Accelerated gradient methods based on numerical discretizations of 
the dynamical system~\eqref{eq:intro-NAG-system} with $f\in\mathcal S_{\mu,L}^{1,1}$
are then considered and analyzed by means of 
a discrete version of the 
Lyapunov function~\eqref{eq:Lt-intro}.
It will be shown that the
implicit scheme (see \eqref{eq:im-ode-NAG}) possesses
linear convergence rate as long 
as the time step size is uniformly bounded below.
This matches the exponential decay 
rate~\eqref{eq:conv-Lt-intro}
in the continuous level. Also, for convex case $\mu=0$, this implicit method amounts to an accelerated PPA,  that is very close to G\"{u}ler's PPA~\cite{guler_new_1992} and enjoys the same rate $O(1/k^2)$ (cf. Theorem \ref{thm:appa}). In Section \ref{sec:GS-correc}, for semi-implicit schemes 
with suitable corrections (either an extrapolation or a gradient step), 
we prove the following convergence rate 
\begin{equation}\label{eq:acc-rate}
	O\left(\min\big\{L/k^2,(1+\sqrt{\mu/L})^{-k}\big\}\right),
\end{equation}
which is optimal in the sense of \cite{Nesterov:2013Introductory}.
Moreover, we can recover Nesterov's optimal 
method~\cite{Nesterov1983,Nesterov:2013Introductory} {\it exactly}
from a semi-implicit scheme with gradient descent correction; see Section \ref{sec:NAG}. Therefore, instead of using estimate sequence, our ODE approach provides an alternative derivation of Nesterov's method and hopefully more intuitive for understanding the acceleration mechanism. From this point of view, we name both~\eqref{eq:intro-NAG-system} and~\eqref{eq:ode-NAG-intro} as {\em Nesterov accelerated gradient (NAG) flow}.

As a proof of concepts, we also generalize our NAG flow to the composite optimization problem
\begin{equation}\label{eq:comp-min}
	\min_{x\in Q} f(x):=
	\min_{x\in Q} \left [ h(x)+g(x) \right ],
\end{equation}
where $Q\subseteq V$ is a (simple) closed convex set, $h\in\mathcal S_{\mu,L}^{1,1}(Q)$ with $0\leqslant \mu\leqslant L<\infty$ 
and $g:V\to\mathbb R\cup\{+\infty\}$ is proper, closed and convex. As usual, we use $\dom g$ to denote the effective domain of $g$ and assume that $Q\cap {\bf dom}\, g\neq\emptyset$.
Treating~\eqref{eq:comp-min} as an unconstrained minimization 
of $F=f+i_Q$ where $i_Q$ denotes the indicator function of $Q$, the generalized version of~\eqref{eq:ode-NAG-intro} is a second-order differential inclusion
\begin{equation}\label{eq:g-acflow}
	\gamma x''+
	\left(\mu+\gamma\right)x'
	+\partial F(x)\ni 0.
\end{equation}
We shall give the solution existence of~\eqref{eq:g-acflow} in proper 
sense and then obtain the exponential decay \eqref{eq:conv-Lt-intro} for almost all $t>0$.

For the unconstrained case $Q = V$, by using the tool of composite gradient mapping~\cite[Chapter 2]{Nesterov:2013Introductory}, a semi-implicit scheme with correction for the generalized NAG flow~\eqref{eq:g-acflow} is presented and leads to an accelerated proximal gradient method (APGM); see Algorithm \ref{algo:APG}. We also give a simplified variant that is closely related to FISTA~\cite{Beck2009}. For the constrained problem~\eqref{eq:comp-min}, an accelerated forward-backward method is proposed in Algorithm \ref{algo:AFB}. Both two algorithms call the proximal operation of $g$ (over $Q$) only once in each iteration, and they are proved to share the same convergence rate~\eqref{eq:acc-rate}.

The rest of this paper is organized as follows. 
In the continuing of the introduction, 
we will review some existing works devoting to  
the accelerated gradient methods from the ODE point of view. 
Next, in Section \ref{sec:stab}, we shall explain the acceleration mechanism
from $A$-stability theory of ODE solvers and derive our NAG flow model as well. 
Then in Section \ref{sec:acceleratedflow} we 
focus on the NAG flow and prove its exponential decay. 
After that, accelerated gradient methods based on numerical discretizations 
are proposed and analyzed in Sections \ref{sec:im}, \ref{sec:GS-correc} and \ref{sec:NAG}. 
Finally, in Section \ref{sec:comp}, 
we extend the our NAG flow to 
composite optimization 
and propose two new accelerated methods
with convergence rate analysis. 
\subsection{Related works}
The well-known momentum method can be traced back to 1960s. In~\cite{polyak_methods_1964}, 
Polyak studied the heavy ball (HB) method 
\begin{equation}
	\label{eq:heavy ball}
	x_{k+1} =x_k -\alpha\nabla f(x_k)+\beta(x_k -x_{k-1})
\end{equation}
and its continuous analogue, the 
heavy ball dynamical system:
\begin{equation}
	\label{eq:heavy ball-ode}
	x''+\alpha_1 x'+\alpha_2\nabla f(x) = 0.
\end{equation}
Local linear convergence 
results for~\eqref{eq:heavy ball} and~\eqref{eq:heavy ball-ode} 
via spectrum analysis were established in \cite[Theorem 9]{polyak_methods_1964}. Note that the HB method~\eqref{eq:heavy ball} 
adds a momentum term up to
the gradient step and is sensitive to its parameters. For $f\in\mathcal S_{\mu,L}^{1,1}$, it shares the same theoretical convergence rate \eqref{eq:rate-gd} as the gradient descent method; see \cite{ghadimi_global_2015,sun_non-ergodic_2018}. To our best knowledge, no work has established the global accelerated rate \eqref{eq:acc-rate} for the original HB method \eqref{eq:heavy ball}. Recently, Nguyen et al.~\cite{nguyen_accelerated_2018} developed the so-called 
accelerated residual method which
combines~\eqref{eq:heavy ball} 
with an extra gradient descent step:
\[
\left\{
\begin{split}
	{}&y_{k} =x_k -\alpha\nabla f(x_k)+\beta(x_k -x_{k-1}),\\
	{}&x_{k+1} = y_k-\frac{\alpha}{\beta+1}\nabla f(y_k).
\end{split}
\right.
\]
Numerically, 
they verified the efficiency and usefulness 
of this method with a restart strategy. 
We refer to~\cite{alvarez_minimizing_2000,attouch_heavy_2000,balti_asymptotic_2016,goudou_gradient_2009} 
for further investigations of the HB system~\eqref{eq:heavy ball-ode}. 

To understand 
an accelerated gradient method with 
the rate $O(1/k^2)$
proposed by Nesterov~\cite{Nesterov1983},
Su, Boyd and  Cand\`es~\cite{Su;Boyd;Candes:2016differential} 
derived the following second-order ODE
\begin{equation}\label{eq:Su2015-ode}
	x'' + \frac{\alpha}{t}x' + \nabla f(x) = 0,\quad t>0,
\end{equation}
where $\alpha>0$ and $f\in\mathcal S_{0,L}^{1,1}$.
If $\alpha\geqslant 3$ or 
$1<\alpha<3$ and $(f-f(x^*))^{(\alpha-1)/2}$ 
is convex, then they proved the decay rate $O(t^{-2})$.
If $\alpha\geqslant 3$ and $f$ is strongly convex, 
then they also obtained a faster rate $O(t^{-2\alpha/3})$.
Later on, Aujol and Dossal~\cite{aujol_optimal_2017} established a generic result:
\begin{equation}\label{eq:conv-Su2015-ode-r}
	f(x(t)) - f(x^*) \leqslant 
	\left\{
	\begin{aligned}
		&Ct^{-2},&&\text{if}~\alpha\geqslant 2\beta+1,\\
		&Ct^{-2\alpha/(2\beta+1)},&&\text{if}~0<\alpha<2\beta+1,
	\end{aligned}
	\right.
\end{equation}
where $\beta>0$ and $(f-f(x^*))^\beta$ is convex. 
Almost at the same time, Attouch et al.~\cite{Attouch:2019} obtained the estimate~\eqref{eq:conv-Su2015-ode-r} 
for $\beta=1$ and 
considered numerical discretizations for~\eqref{eq:Su2015-ode} 
with the convergence rate $O(k^{-\min\{2,2\alpha/3\}})$,
which matches the continuous decay property
\eqref{eq:conv-Su2015-ode-r} for the case $\beta=1$.
Also, Vassilis et al.
\cite{Vassilis2018} studied the non-smooth 
version of~\eqref{eq:Su2015-ode}:
\begin{equation}\label{eq:Su2015-ode-nonsmooth}
	x'' + \frac{\alpha}{t}x' + \partial f(x) \ni 0.
\end{equation}
They proved that the solution trajectory of 
\eqref{eq:Su2015-ode-nonsmooth}
converges to a minimizer of $f$ and 
derived the decay estimate~\eqref{eq:conv-Su2015-ode-r} for $\beta=1$. 
For more works and generalizations 
related to the model~\eqref{eq:Su2015-ode} and the corresponding algorithms, we refer to 
\cite{apidopoulos_convergence_2018,attouch_fast_2015,Attouch_2016-fast,attouch_convergence_2018,cabot_long_2009} and references therein.  

Recently, Wibisono et al.
\cite{Wibisono;Wilson;Jordan:2016variational} 
introduced a Lagrangian 
\begin{equation}\label{eq:Lagrangian}
	\mathcal E(y,w,t) =  \frac{e^{\int_{0}^{t}\alpha(s)\,\mathrm ds}}{\alpha(t)\beta(t)}
	\left(
	\frac{\beta(t)}{2}\nm{w}^2-\alpha^2(t)f(y)
	\right),
\end{equation}
for smooth and convex $f$, where the scaling function 
$\alpha:\mathbb R_+\to\mathbb R_+$ 
is continuous and $\beta:\mathbb R_+\to\mathbb R_+ $ satisfies
\begin{equation}\label{eq:beta-jordan}
	\beta'\geqslant -\alpha\beta,\quad \beta(0)=\beta_0>0.
\end{equation}
The Lagrangian~\eqref{eq:Lagrangian} itself
introduces a variational problem, the Euler--Lagrange 
equation to which is 
\begin{equation}\label{eq:ode-jordan}
	\left\{
	\begin{split}
		{}&			y' =\alpha(w-y),\\
		{}&			\beta w'=-\alpha 
		\nabla f(y).
	\end{split}
	\right.
\end{equation}
They then established the convergence rate (cf. 	\cite[Theorem 2.1]{Wibisono;Wilson;Jordan:2016variational})
\begin{equation}\label{eq:conv-Wibisono-ode}
	f(y(t))-f(x^*)\leqslant 
	e^{-\int_{0}^{t}\alpha(s)\,\mathrm ds}
	\mathcal L(0),
\end{equation}
by means of the Lyapunov function
\begin{equation*}
	\mathcal L(t)=e^{\int_{0}^{t}\alpha(s)
		\,\mathrm ds}\big [f(y(t))-f(x^*) \big ]+
	\frac{1}{2}\nm{w(t)-x^*}^2.
\end{equation*}

Following this work, for any $f\in\mathcal S_{\mu}^1$ with $\mu>0$, 
Wilson et al.~\cite{Wilson:2018} 
introduced another Lagrangian
whose Euler--Lagrange equations reads as
\begin{equation}\label{eq:Wilson-ode}
	\left\{
	\begin{split}
		{}&			y' = \alpha(w-y),\\
		{}&			\mu w'=\mu\alpha(y-w)-\alpha\nabla f(y),
	\end{split}
	\right.
\end{equation}
with the same scaling function
$\alpha$ in~\eqref{eq:Lagrangian}.
They proved the decay estimate
\eqref{eq:conv-Wibisono-ode} as well, 
by using the Lyapunov function
\begin{equation}\label{eq:Lt-WRJ}
	\mathcal L(t)=
	e^{\int_{0}^{t}\alpha(s)\,\mathrm ds}
	\left [
	f(y(t)) - f(x^*) + \frac{\mu}{2}\nm{w(t) - x^*}^2\right ].
\end{equation}
When $\alpha=\sqrt{\mu}$, \eqref{eq:Wilson-ode}
gives the following model
\begin{equation}    \label{eq:Wilson-ode-par}
	y''+2\sqrt{\mu}y'+\nabla f(y) = 0,
\end{equation}
which reduces to 
an HB system (cf.~\eqref{eq:heavy ball-ode}).
From another
Lyapunov function 
\[
\mathcal L(t) =
f(y(t)) - f(x^*)  +    
\frac{\mu}{2}\nm{y(t)+y'(t)/\alpha(t)-x^*}^2,
\]	
Siegel~\cite{Siegel:2019} also derived 
\eqref{eq:Wilson-ode-par} 
and proved that
\begin{equation*}
	f(y(t)) - f(x^*) 
	\leqslant 2e^{-\sqrt{\mu}t}
	\big [f(y(0))-f(x^*) \big ].
\end{equation*}
In addition, Siegel~\cite{Siegel:2019} and 
Wilson et al.~\cite{Wilson:2018} 
proposed two semi-explicit schemes for 
\eqref{eq:Wilson-ode-par} individually. 
Both of their schemes are 
supplemented with an extra
gradient descent step and share the same 
linear convergence rate $O((1-\sqrt{\mu/L})^k)$.

Recently, introducing the so-called duality gap which is the difference of appropriate upper and lower bound approximations for
the objective function, Diakonikolas and Orecchia~\cite{diakonikolas_approximate_2018} presented a general framework 
for the construction and analysis of continuous time dynamical systems 
and the corresponding numerical discretizations. 
They recovered several 
existing ODE models such as the gradient flow \eqref{gradientflow}, 
the mirror descent dynamic system and its accelerated version. 
We mention that the derivation of our NAG model and 
analyses of discrete algorithms are fundamentally different from 
their duality gap technique.	
\section{Stability of ODE Solvers and Acceleration}
\label{sec:stab}		
In what follows, for any square matrix $M\in\mathbb R^{d\times d}$, $\sigma(M)$ denotes the spectrum of $M$, i.e., the set of all eigenvalues of $M$. The spectral radius is then defined by $\rho(M) := \max_{\lambda \in \sigma(M)} |\lambda|$, and when $M$ is invertible, its condition number $\kappa(M) := \rho(M^{-1})\rho (M)$. If $\sigma(M)\subset\mathbb R$, then $\lambda_{\min}(M)$ and $\lambda_{\max}(M)$ stand for the minimum and maximum of $\sigma(M)$, respectively.  Moreover, $\nm{\cdot}_2$ is the usual $2$-norm for vectors and matrices.

To present our main idea as simple as possible, in this section, unless other specified, we restrict ourselves to the quadratic objective $f(x) = \frac{1}{2}x^{\top}Ax$, where $A$ is a symmetric matrix with the bound
\[
0\leqslant \mu:=\lambda_{\min}(A)\leqslant \lambda \leqslant   \lambda_{\max}(A):=L\quad\forall\,\lambda\in\sigma(A).
\]
For this model example, $\nabla f(x) = Ax$ and the gradient flow \eqref{gradientflow} reads as $x'=-Ax$. The global minimal is achieved at $x^*=0$, and when $\mu >0$, the condition number of $A$ is $\kappa(A)=L/\mu$. 

\subsection{A-stability of ODE solvers}
Let $G\in\mathbb R^{d\times d}$ and assume $\mathfrak{Re}(\lambda) < 0$ for all $\lambda
\in \sigma(G)$. For the linear ODE system
\begin{equation}\label{eq:G-ODE}
	y ' = G y, \quad y(0) = y_0\in\mathbb R^{d},
\end{equation}
it is not hard to derive that $\nm{y(t)}_2\to0$ as $t\to\infty$ (see~\cite[Theorem 7]{bellman_stability_1953} for instance). 
Hence $y^*=0$ is an
equilibrium of the dynamic system~\eqref{eq:G-ODE}.

We now recall the concept of $A$-stability of ODE solves~\cite{LeVeque:2007Finite,Suli:2010Numerical}. 
A one-step method $\phi$ for~\eqref{eq:G-ODE} 
with step size $\alpha>0$ 
can be formally written as
\begin{equation}\label{eq:1-step}
	y_{k+1} = E_{\phi}(\alpha, G) y_{k}.
\end{equation}
As $y^* = 0$ is an equilibrium point,~\eqref{eq:1-step} also gives the error equation.
The scheme $\phi$ is called {\it absolute stable} or {\it $A$-stable} if $\rho (E_{\phi}( \alpha, G)) < 1$ from which the asymptotic convergence  $y_{k} \to 0$ follows (cf.~\cite[Theorem 6.1]{Demmel:1997Applied}). 		
If $\rho (E_{\phi}( \alpha, G))  < 1$ holds for all $\alpha>0$, then it is called {\it unconditionally $A$-stable}, and if $\rho (E_{\phi}( \alpha, G))  < 1$ for any $\alpha\in I$, where $I$ is an interval of the positive half line, then the scheme is called {\it conditionally $A$-stable}.

If $E_{\phi}(\alpha,G)$ is normal, then $\| E_{\phi}(\alpha,G) \|_2=\rho(E_{\phi}(\alpha,G))$. Therefore for $A$-stable methods the linear convergence 
follows directly from
the norm contraction
\begin{equation}\label{eq:bd-by-rho}
	\nm{y_{k+1}}_2\leqslant 
	\rho(E_{\phi}(\alpha,G))\nm{y_{k}}_2.
\end{equation}
In general cases, however, bounding the spectral radius by one does not imply the norm contraction, i.e.,~\eqref{eq:bd-by-rho} may not be true when $E_{\phi}(\alpha,G)$ is non-normal, even if \eqref{eq:1-step} is $A$-stable. Nevertheless, we shall continue using the tool of $A$-stability through spectral analysis and comment on its limitation in Section \ref{sec:limit}.

\subsection{Implicit and Explicit Euler methods}	\label{sec:GD}
It is well known that the implicit Euler (IE) method 
\[
\frac{y_{k+1}-y_k}{\alpha} = Gy_{k+1}
\]
is unconditionally $A$-stable. Indeed, $E_{\rm IE} ( \alpha, G) = (I - \alpha G)^{-1}$ and $\rho(E_{\rm IE} ( \alpha, G))<1$ for all $\alpha>0$ since all eigenvalues of $\alpha G$ lie on the left of 
the complex plane and their distance to $1$ is larger than one. 
Moreover, as it has no restriction on the step size, the implicit Euler method can achieve faster convergent rate by time rescaling which is equivalent to chose a large step size.

The explicit Euler method
\begin{equation}\label{eq:explicit}
	\frac{y_{k+1}-y_k}{\alpha} = Gy_{k}
\end{equation}
is only conditionally $A$-stable. Let us consider the case $G=-A$ with $\mu>0$. Then~\eqref{eq:explicit} is exactly the gradient descent method for minimizing $\frac{1}{2}x^{\top}Ax$. It is not hard to obtain that
\begin{equation}\label{eq:rho-GD}
	\rho(E_{\rm GD}(\alpha, -A) ) 
	= \rho(I-\alpha A) =
	\max\big\{
	\snm{1-\alpha\mu},
	\,
	\snm{1-\alpha   L}
	\big\}.
\end{equation}
Hence $\rho(E_{\rm GD}(\alpha, -A) ) <1$ provided 
$0<\alpha<2/  L$. Thanks to the symmetry of $A$, we have
$\| E_{\rm GD}(\alpha, -A) \|_2 = \rho(E_{\rm GD}(\alpha, -A) ) $ 
and the norm convergence with linear rate follows.
Moreover, based on~\eqref{eq:rho-GD}, 
a standard argument outputs  the optimal 
choice $\alpha^*  = 2/(\mu+  L)$, 
which gives the minimal spectrum
\begin{equation}\label{eq:conv-opt}
	\| E_{\rm GD}(\alpha^*, -A) \|_2=
	\min_{\alpha >0}
	\rho(I-\alpha A ) 
	=\frac{\kappa(A) -1}{\kappa(A) + 1}.
\end{equation}
A quasi-optimal but simpler choice is $\alpha_*  = 1/ L$ which yields
\begin{equation}\label{eq:conv-qopt}
	\| E_{\rm GD}(\alpha_*, -A) \|_2 
	= \rho(I-\alpha_* A)
	= 1 - \frac{1}{\kappa(A)}.
\end{equation}

We formulate the convergence rates~\eqref{eq:conv-opt} and~\eqref{eq:conv-qopt} in terms of the condition number $\kappa(A)$ as it is invariant to the rescaling of $A$, i.e., $\kappa(cA) = \kappa (A)$ for any real number $c\neq 0$. To be $A$-stable, one has to choose $0<\alpha<2/\lambda_{\max}(A)$. It seems that a simple rescaling to $cA$ can reduce $\lambda_{\max}(cA)$ and thus enlarge the range of the step size. However, the condition number $\kappa(cA) = \kappa (A)$ is invariant. From this we see that for the GD method~\eqref{eq:explicit}, the simple rescaling $cA$ is in vain. 

The magnitude of the step size is relative to $\min |\lambda(G)|$. To fix the discussion, we chose $G = - A/\mu$ in~\eqref{eq:explicit} so that $\lambda_{\min}(A/\mu) = 1$.  Then in order for the explicit Euler method to be $A$-stable it is equivalent to choose $\alpha = O(1/\kappa(A))$ which leads to the contraction rate $1-1/\kappa(A)$. 
Consequently for ill-conditioned problems, tiny step size proportional to $1/\kappa(A)$ is required.

Rather than the rescaling, our main intuition is to seek some transformation $G$ of $A$, that keeps $\min |\lambda(G)| = 1$ and reduces $\kappa(A)$ to $\kappa(G)=O(\sqrt{\kappa(A)})$. 
We wish to construct explicit $A$-stable methods which can enlarge the step size from $O(1/\kappa(A))$ to $O(1/\sqrt{\kappa(A)})$ and consequently improve the contraction rate from $1 - 1/\kappa(A)$ to $O(1 - 1/\sqrt{\kappa(A)})$. 

\subsection{Transformation to the complex plane}		
\label{sec:ODE}					
Let us first consider the case $\mu>0$ and embed $A$ into some $2\times 2$ block matrix $G$ with a rotation built-in. 
Specifically, we construct two candidates
\begin{equation}\label{eq:G-NAG}
	G_{_{\rm HB}} = 
	\begin{pmatrix}
		0 & I\\
		-A/\mu &\quad - 2I
	\end{pmatrix}
	\quad\text{and}\quad
	G_{_{\rm NAG}} = 
	\begin{pmatrix}
		-I & I\\
		I-A/\mu  &\quad -I
	\end{pmatrix}.
\end{equation}
Due to the asymmetrical fact, $\sigma(A)$ will be transformed from the real line to the complex plane. This may shrink the condition number; see the following result.
\begin{proposition}\label{prop: 4G}
	For $G=G_{_{\rm HB}}$ or $G_{_{\rm NAG}}$ given in~\eqref{eq:G-NAG}, it satisfies $\mathfrak{Re}(\lambda) < 0$ for any $\lambda \in \sigma(G)$, which promises the decay property $\nm{y(t)}_2\to0\,$ for the system $y'=Gy$. Moreover, we have
	$\kappa(G_{_{\rm HB}}) = 
	\kappa(G_{_{\rm NAG}}) = \sqrt{\kappa(A)}$.
\end{proposition}			
\begin{proof}
	Let us first consider $G=G_{_{\rm HB}}$. As $A$ is symmetric, we can write $A = U\Lambda U^{\top}$ with unitary matrix $U$ and diagonal matrix $\Lambda$ consisting of eigenvalues of $A$. By applying the similar transform to $G$ with the block diagonal matrix ${\rm diag}(U, U)$, it suffices to consider eigenvalues of
	\[
	R_{_{\rm HB}}=  \begin{pmatrix}
		0 & 1\\
		-\theta& \quad -2
	\end{pmatrix}, \quad \theta\in \sigma(A/\mu).
	\]
	It is clear that $\det R_{_{\rm HB}}   = \theta$ and ${\rm tr}\, R_{_{\rm HB}}=-2<0$. In addition, since $\snm{{\rm tr}\,R_{_{\rm HB}} }^2\leqslant {}4\det R_{_{\rm HB}}$, any eigenvalue $\lambda_R\in \sigma(R_{_{\rm HB}})$ is a complex number and 
	\[
	\mathfrak{Re}(\lambda_R) {}=-1,\quad 
	\snm{\lambda_R} = {}\sqrt{\det R_{_{\rm HB}}}=\sqrt{\theta}.
	\]		
	As $1 = \lambda_{\min}(A/\mu)\leqslant
	\theta\leqslant 
	\lambda_{\max}(A/\mu) = \kappa(A),$
	we conclude $\kappa(G_{_{\rm HB}}) = \sqrt{\kappa(A)}$. 
	
	Apply the similar transformation with	$P = 
	\begin{pmatrix}
		1&\; 0\\1&\; 1
	\end{pmatrix},$
	we observe that
	\[
	R_{_{\rm NAG}}  = PR_{_{\rm HB}}P^{-1} = 
	\begin{pmatrix}
		-1 & 1\\
		1-\theta  &\quad -1
	\end{pmatrix}.	
	\]
	So $\sigma(R_{_{\rm NAG}} ) = \sigma(R_{_{\rm HB}})$ and consequently $\kappa(G_{_{\rm NAG}}) = \sqrt{\kappa(A)}$. This completes the proof of this proposition.
	 \end{proof}
\medskip

We write $y = (x, v)^{\top}$ and eliminate $v$ in $y'=Gy$ to get a second order ODE of $x$, in which we replace $Ax$ by general form $\nabla f(x)$. Both $G_{_{\rm HB}}$ and $G_{_{\rm NAG}} $ yield the same ODE
\begin{equation}\label{eq:ode-HB-NAG}
	\mu x'' + 2\mu x' + \nabla f(x) = 0,
\end{equation}	
which is a special case of the HB model 
(cf.~\eqref{eq:heavy ball-ode}).

Note that we can find a lot of transformations $G$ and derive corresponding ODE models. Indeed, given any $G$ that meets our demand, both $cG$ and $QGQ^{-1}$ are acceptable candidates, where $c>0$ and $Q$ is some invertible matrix. We are not going further deep beyond those two transformations given in \eqref{eq:G-NAG} for the strongly convex case $\mu>0$ but aim to combine the transformation with refined time scaling to propose another one for convex case $\mu=0$ in Section \ref{sec:convx}.  
\subsection{Acceleration from a Gauss-Seidel splitting}					
\label{sec:GS}
We now consider numerical discretization for \eqref{eq:G-ODE} with $G=G_{_{\rm HB}}$ and $G_{_{\rm NAG}}$ given in~\eqref{eq:G-NAG}. As discussed in Section \ref{sec:GD}, the implicit Euler method is unconditional $A$-stable. But computing $(I - \alpha G)^{-1}$ needs significant effort and may not be practical. 

One may hope that the explicit Euler method $y_{k+1} = (I + \alpha G) y_k$ will be $A$-stable with step size $\alpha = O(1/\kappa(G))= O(1/\sqrt{\kappa(A)})$. Unfortunately, unlike the discussion for \eqref{eq:explicit} with $G=-A$, where $\sigma(I-\alpha A)$ lies on the real line and $\rho(I-\alpha A)$ can be easily shrunk by choosing $\alpha = 1/\rho(A)$ (cf. \eqref{eq:rho-GD}), the general asymmetric $G$ spreads the spectrum on the complex plane. For both $G=G_{_{\rm HB}}$ and $G=G_{_{\rm NAG}}$, we have $\Re(\lambda) = -1$ for all $\lambda\in\sigma(G)$. Denote by $r = \rho(G)$. Then $\rho^2(I + \alpha G) = (1-\alpha)^2 + \alpha^2 (r^2-1)$. To be $A$-stable, requiring $\rho(I + \alpha G) < 1$ is equivalent to letting $0< \alpha < 2/r^2 = O(1/\kappa(A))$, where small step size $\alpha = O(1/\kappa(A))$ is still needed. The optimal choice $\alpha^*=r^{-2}$ only gives 
\[
\rho(I+\alpha^*G) = 1-\alpha^* = 1-O(1/\kappa(A)),
\]
where no acceleration has been obtained.

We then expect that an explicit scheme closer to the 
implicit Euler method will hopefully have better stability with larger step size. 
Motivated by the Gauss-Seidel (GS) method~\cite{saad_iterative_2003} for computing $(I - \alpha G)^{-1}$, we consider the matrix splitting $G = M + N$ with $M $ being the lower triangular part of $G$ (including the diagonal part) and $N = G - M$, and propose the following Gauss-Seidel splitting scheme
\begin{equation}\label{eq:GS}
	\frac{y_{k+1} - y_{k}}{\alpha} = M y_{k+1} + N y_{k}
\end{equation}
which gives the relation
\begin{equation}\label{eq:EalphaG}
	y_{k+1} = E(\alpha, G) y_{k},\quad 
	E(\alpha, G) :=(I - \alpha M)^{-1} (I + \alpha N).
\end{equation}
Note that for $G=G_{_{\rm HB}}$ and $G_{_{\rm NAG}}$, the scheme~\eqref{eq:GS} is still explicit as the lower triangular block matrix $I - \alpha M$ can be inverted easily, without involving $A^{-1}$. 

The spectrum bound is given below and for the algebraic proof details, we refer to Appendix \ref{app:proof}.
\begin{theorem}\label{thm:spec-bd-GS}
	For $G = G_{_{\rm HB}}$ or $G_{_{\rm NAG}}$ 
	given in~\eqref{eq:G-NAG}, if $				0< \alpha  \leqslant  2/\sqrt{\kappa(A)}$, then the Gauss-Seidel splitting 
	scheme~\eqref{eq:GS} is $A$-stable and 
	\[
	\rho(E(\alpha,G)) \leqslant   \frac{1}{\sqrt{1+2\alpha}}.
	\]
\end{theorem}
\subsection{Dynamic time rescaling for the convex case} 
\label{sec:convx}
The ODE model~\eqref{eq:ode-HB-NAG} given in Section \ref{sec:ODE} cannot treat the case $\mu=0$ and the previous spectral analysis also fails. Equivalently the condition number $\kappa(A)$ is infinity and the spectrum bound becomes $1$. To conquer this, a careful rescaling is needed. Throughout this subsection, we assume $\mu = 0$. 

For the gradient flow
\begin{equation}
	\label{eq:gd-flow}
	x'(t) = -\nabla f(x(t)),
\end{equation}
one can easily establish the sub-linear rate $f(x(t))\leqslant C/t$; see~\cite{Su;Boyd;Candes:2016differential}.
To recover the exponential rate, we introduce a time rescaling $t(s) =e^{s}$ and let $y(s)=x(t(s))$. Then~\eqref{eq:gd-flow} becomes the rescaled gradient flow
\begin{equation}\label{eq:re-gd-flow}
	\gamma(s) y'(s) = -\nabla f(y(s)),
\end{equation}
with the scaling factor $\gamma(s)=e^s$.
Besides, the previous sublinear rate $f(x(t))\leqslant C/t$ turns into $f(y(s)) \leqslant Ce^{-s}$.
That is in the continuous level, we can achieve exponential decay 
through suitable rescaling of time even for convex case $\mu =0$ .

Now let us go back to our model case $f(x) = \frac{1}{2}x^{\top}Ax$ with $\mu=0$ and $\lambda_{\max}(A) = L$. Coupled with the transformation $G_{_{\rm NAG}} $, we consider 
\begin{equation}\label{eq:Gy}
	y' = G(\gamma) \, y,\quad G(\gamma) = 
	\begin{pmatrix}
		- I & \quad I \\
		-A/\gamma & \quad O
	\end{pmatrix},
\end{equation}
where $y = (x, v)^{\top}$ and 
\begin{equation}\label{eq:ode-gama}
	\gamma' = -\gamma, \quad \gamma(0)=\gamma_0>0.
\end{equation}
This gives a second-order ODE in terms of $x$:
\begin{equation}\label{eq:nag-ODE}
	\gamma  x'' + \gamma x'  + \nabla f(x)=0,
\end{equation}
which is in the HB type but with variable damping coefficients.

Obviously, the implicit Euler method for solving~\eqref{eq:Gy} is still unconditional $A$-stable. 
We now apply the GS splitting~\eqref{eq:GS} to~\eqref{eq:Gy} 
and get 
\begin{equation}\label{eq:split-1}
	y_{k+1} = {}E(\alpha_{k}, G(\gamma_{k+1})) y_{k},
\end{equation}
where $E(\alpha_{k}, G(\gamma_{k+1}))$ is defined in~\eqref{eq:EalphaG}.
The equation \eqref{eq:ode-gama} is discretized by that
\begin{equation}\label{eq:im-gk}
	\gamma_{k+1}={}\gamma_{k}-\alpha_{k}\gamma_{k+1}.
\end{equation}
Eliminating $v_{k}$ in~\eqref{eq:split-1} will 
give an HB method with variable coefficients 
$$
x_{k+1} = x_{k} -\frac{\alpha_k\alpha_{k-1}}
{\gamma_k+\alpha_k\gamma_k}\nabla f(x_{k})  + \frac{\alpha_k}{\alpha_{k-1}+\alpha_k
	\alpha_{k-1}} (x_{k} - x_{k-1}). 
$$

Instead of studying the spectrum bound 
$E(\alpha_k,G(\gamma_{k+1}))$ which is $1$, we 
apply the scaling technique to obtain a regularized 
matrix
\[
\widetilde{E}_{k} =
\begin{pmatrix}
	I & \;  O \\
	O & \; \gamma_{k+1} I
\end{pmatrix}
E(\alpha_{k}, G(\gamma_{k+1}))
\begin{pmatrix}
	I & \; O \\
	O & \;  \gamma_{k} I
\end{pmatrix}^{-1},
\]
which is nearly similar with $E(\alpha_k,G(\gamma_{k+1}))$. 
Set $z_k=\begin{pmatrix}
	I & \; O \\
	O & \; \gamma_{k} I
\end{pmatrix}
y_{k}
$, then the discrete system~\eqref{eq:split-1} 
for $\{y_k\}$ becomes
\begin{equation}\label{eq:zk}
	z_{k+1}=\widetilde{E}_{k}z_k,
\end{equation}

With a careful chosen step size, the spectrum bound of $\widetilde{E}_{k}$ is given below and for the algebraic proof details, we refer to Appendix \ref{app:proof}. We note that, the step size choice in Theorem \ref{thm:spec-mu-0} is only to agree with the setting of Lemma \ref{lem:gk-ak-im} and for general choice $L\alpha_k^2/\gamma_k=O(1)$ and suitable initial value $\gamma_0$, it is possible to maintain the spectrum bound~\eqref{eq:rho} together with the decay estimate~\eqref{eq:rho-prod}.
\begin{theorem}\label{thm:spec-mu-0}
	If $\gamma_0=L$ and $L\alpha_{k}^2 =\gamma_{k}(1+\alpha_k)$, then both the 
	scheme~\eqref{eq:split-1} and its equivalent form~\eqref{eq:zk} are $A$-stable 
	and we have
	\begin{equation}\label{eq:rho}
		\rho (\widetilde E_{k}) =\frac{\gamma_{k+1}}{\gamma_{k}}
		=\frac{1}{1+\alpha_{k}},
	\end{equation}
	which further implies that
	\begin{equation}\label{eq:rho-prod}
		\prod_{i=0}^{k-1}\rho(\widetilde{E}_i) 
		= \frac{\gamma_{k}}{\gamma_0} = O(k^{-2}).
	\end{equation}
\end{theorem}
\subsection{Limitation of spectral analysis}
\label{sec:limit}
For quadratic objective $f$, both the ODE models~\eqref{eq:ode-HB-NAG} 
and~\eqref{eq:nag-ODE} are linear and the spectrum bound of $E(\alpha,G)$ for the Gauss-Seidel splitting~\eqref{eq:EalphaG} is derived.
But as pointed out in the beginning, for $A$-stable methods, bounding the spectral radius by one is not sufficient 
for the norm convergence if the matrix $E(\alpha,G)$ is non-normal; 
see convincible examples 
in~\cite[Appendix D.2]{LeVeque:2007Finite} and~\cite[Appendix D.4]{LeVeque:2007Finite}.

Moving beyond quadratic $f$ and nonlinear ODE systems, transient growth or instability of perturbed problems can easily lead to nonlinear instabilities. Particularly, for the HB system \eqref{eq:heavy ball-ode}, it is shown in~\cite{LessardRechtPackard2016} that the parameters optimized for linear ODE models does not guarantee the global convergence for a nonlinear system.

To provide rigorous convergence analysis for both continuous and discrete levels,  in the sequel we shall introduce the tool of Lyapunov function. Following many related works~\cite{Attouch_2016-fast,Su;Boyd;Candes:2016differential,Wibisono;Wilson;Jordan:2016variational}, we first analyze some proper ODEs via a Lyapunov function, then construct optimization algorithms from numerical discretizations of continuous models and use a discrete Lyapunov function to establish the convergence rates of the proposed algorithms. 
\section{Nesterov Accelerated Gradient Flow}
\label{sec:acceleratedflow}
\subsection{Continuous problem}\label{sec:NAG-ode}
In the previous section, we have obtained two ODE models for 
quadratic objective $f(x) = \frac12x^{\top}Ax$ with $\mu > 0$ and $\mu = 0$, respectively.
To handle those two cases in a unified way, we combine $G_{_{\rm NAG}}$ in~\eqref{eq:G-NAG} with $G(\gamma)$ in~\eqref{eq:Gy} and consider a new transformation 
\begin{equation}\label{eq:G}
	G = 
	\begin{pmatrix}
		- I & \quad I \\
		\mu/\gamma -A/\gamma & \quad -\mu/\gamma \, I
	\end{pmatrix},
\end{equation}
where 	
\begin{equation}\label{eq:gammat-NAG}
	\gamma' = \mu-\gamma,\quad\gamma(0)=\gamma_0>0.
\end{equation}
One can solve the above equation and obtain
\[
\gamma(t) =
\mu + (\gamma_0-\mu)
{\mathrm e}^{-t},
\quad t\geqslant 0.
\]
Since $\gamma_0>0$, we have that 
$\gamma(t)>0$ for all $t\geqslant 0$ and 
$\gamma(t)$ converges to $\mu$ 
exponentially and monotonically as $t \to +\infty$. In particular, if $\gamma_0=\mu>0$, then $\gamma(t)=\mu$. 
Therefore, when $\mu=0$, ~\eqref{eq:G} reduces to~\eqref{eq:Gy} and when $\gamma_0=\mu>0$, ~\eqref{eq:G} recovers 
\eqref{eq:G-NAG} indeed. Correspondingly, the transform~\eqref{eq:G} gives 
the system
\begin{equation}\label{eq:NAG-quad}
	\left\{
	\begin{aligned}
		x'={}& v - x,\\
		\gamma v'={}&\mu (x - v ) - 
		Ax.
	\end{aligned}
	\right.
\end{equation}

Heuristically, for general $f\in\mathcal S_\mu^1$ with $\mu\geqslant 0$, 
we just replace $Ax$ in~\eqref{eq:NAG-quad} 
with $\nabla f(x)$ and obtain our NAG flow
\begin{equation}\label{eq:ode-sys-NAG}
	\left\{
	\begin{aligned}
		x'={}& v - x,\\
		\gamma v'={}&\mu (x - v ) - 
		\nabla f(x),
	\end{aligned}
	\right.
\end{equation}
with initial conditions $x(0)=x_0$ and $v(0)=v_0$. 
The equivalent second-order ODE (will also be abbreviated as NAG flow) 
reads as follows
\begin{equation}\label{eq:ode-NAG}
	\gamma x''+(\mu+\gamma)x'+\nabla f(x)=0,
\end{equation}
with initial conditions $x(0)=x_0$ and $x'(0)=v_0-x_0$.
Clearly, if $\gamma_0 = \mu>0$, then~\eqref{eq:ode-NAG} becomes~\eqref{eq:ode-HB-NAG}, and if $ \mu=0$, then~\eqref{eq:ode-NAG} 
coincides with~\eqref{eq:nag-ODE}. 

Motivated by~\eqref{eq:Lt-WRJ}, we introduce a
Lyapunov function for~\eqref{eq:ode-sys-NAG}:
\begin{equation}\label{eq:Lt-NAG}
	\mathcal L(t):=
	f(x(t))-f(x^*)+\frac{\gamma(t)}{2}
	\nm{v(t)-x^*}^2,\quad t\geqslant 0.
\end{equation}
In addition, we need the following lemma, which is trivial but 
very useful for the convergence analysis in both of the continuous and discrete levels.
\begin{lemma}\label{lem:cross}
	For any $u,v,w\in V$, we have
	\[
	2(u-v,v-w) = \nm{u-w}^2-\nm{u-v}^2-\nm{v-w}^2.
	\]
\end{lemma}

We first present the well-posedness of~\eqref{eq:ode-NAG} 
and prove the exponential decay property of the Lyapunov function~\eqref{eq:Lt-NAG}.
\begin{lemma}\label{lem:conv-ode-NAG}
	If $f\in\mathcal {S}_{\mu,L}^{1,1}$ with $\mu\geqslant 0$, 
	then the NAG flow~\eqref{eq:ode-NAG} 
	admits a unique solution 
	$x\in C^2([0,\infty); V)$ and moreover
	\begin{equation}\label{eq:dLt}
		\mathcal L'(t)\leqslant
		-\mathcal L(t) - 
		\frac{\mu}{2}
		\|x'(t)\|^{2},
	\end{equation}
	which implies that
	\begin{equation}
		\label{eq:conv-ode-NAG}
		\mathcal L(t)+\frac{\mu}{2}
		\int_{0}^{t}e^{s-t}\nm{x'(s)}^2{\rm d }s
		\leqslant {\mathrm e}^{-t}
		\mathcal L(0),\quad t\geqslant 0.
	\end{equation}
\end{lemma}
\begin{proof}	
	Basically, as $\nabla f$ is Lipschitz continuous, applying 
	the standard existence and uniqueness results of 
	ODE (see~\cite[Theorem 4.1.4]{Ahmad2015}) yields the fact that 
	the system~\eqref{eq:ode-sys-NAG} admits a unique classical solution 
	$(x,v)\in C^1([0,\infty); V)\times C^1([0,\infty); V)$. 
	This implies that $x' =v-x\in C^1([0,\infty); V)$, and therefore 
	$x\in C^2([0,\infty); V)$ is also the unique solution to 
	our NAG flow~\eqref{eq:ode-NAG}.
	
	It remains to prove~\eqref{eq:dLt}, which yields the exponential decay~\eqref{eq:conv-ode-NAG} immediately. A straightforward calculation yields that
	\[
	\mathcal L'(t)={} \dual{\nabla f(x),x'}+
	\frac{\gamma'}{2}\nm{v-x^*}^2+
	\gamma\dual{v', v-x^*},
	\]
	and by~\eqref{eq:gammat-NAG} and~\eqref{eq:ode-sys-NAG}, we replace 
	$\gamma'$ and $v'$ by their right hand side terms and obtain
	\begin{align}
		\mathcal L'(t)= \dual{\nabla f(x),x'}+
		\frac{  \mu-\gamma}{2}\nm{v-x^*}^2+
		\dual{\mu ( x-v)-
			\nabla f(x),v-x^*}.
		\label{eq:key-step}
	\end{align}
	Let us focus on the last term. Thanks to Lemma \ref{lem:cross},
	\begin{equation*}
		\mu(x-v, v-x^{*}) =
		\frac{\mu}{2}\left(			\left\|x-x^{*}\right\|^{2}
		-\|x-v\|^{2}-
		\left\|v-x^{*}\right\|^{2}\right),
	\end{equation*}
	and the gradient term is split as follows
	\begin{equation}\label{cross}
		-\dual{\nabla f(x), v-x^*} =- \dual{ \nabla f(x), v-x} 
		-\dual{\nabla f(x), x-x^*}.
	\end{equation}
	By the relation $x' = v- x$, the first term in 
	\eqref{cross} becomes $\dual{-\nabla f(x), x'}$ which cancels the first term in~\eqref{eq:key-step}. Combining all identities together gives 
	\begin{equation}\label{eq:dL}
		\mathcal L'(t)
		={} \frac{\mu}{2}
		\nm{x-x^{*}}^{2}
		-\dual{\nabla f(x), x-x^*}
		-\frac{ \gamma}{2}\nm{v-x^*}^2-
		\frac{\mu}{2}
		\|x'\|^{2}.
	\end{equation}
	As $f$ is $\mu$-strongly convex (cf.\eqref{eq:convex-mu}), there holds
	\[
	\frac{\mu}{2}
	\nm{x-x^{*}}^{2}
	-\dual{\nabla f(x), x-x^*}
	\leqslant 	f(x^*) - f(x),
	\]
	and plugging this into~\eqref{eq:dL} implies that
	\[
	\mathcal L'(t)\leqslant
	-\mathcal L(t) - 
	\frac{\mu}{2}
	\|x'(t)\|^{2},
	\]
	which proves~\eqref{eq:dLt} and thus completes the proof of this lemma.
	 \end{proof}
\begin{remark}\label{rem:gamma}
	According to the proof of Lemma \ref{lem:conv-ode-NAG}, the equation~\eqref{eq:gammat-NAG} for $\gamma$ can be relaxed to $\gamma' \leqslant\mu-\gamma$. Indeed, this makes~\eqref{eq:key-step} and~\eqref{eq:dL} become inequality but leaves the final estimate~\eqref{eq:dLt} invariant.  
\end{remark}
\subsection{Rescaling property}
\label{sec:NAG-ode-rescale}
Based on our NAG flow~\eqref{eq:ode-sys-NAG} (or~\eqref{eq:ode-NAG}), it is 
possible to use time scaling technique to construct more ODE 
systems with any desirable convergence rate. It is worth distinguishing 
the connection and difference with existing dynamical models.

Specifically, let $\alpha$ be any continuous nonnegative function 
on $\mathbb R_+$, and consider the time rescaling  
\begin{equation}\label{eq:scaling}
	t(\tau)=\int_{0}^{\tau}\alpha(s)\,\mathrm ds,
	\quad\tau>0.
\end{equation}
Set $y(\tau) = x(t(\tau)),w(\tau) = v(t(\tau))$ and $\beta(\tau) = \gamma(t(\tau))$, then it is clear that 
\[
y'(\tau) = t'(\tau)x'(t(\tau)) = \alpha(\tau)x'(t(\tau)),
\]
Similarly, $w'(\tau) = \alpha(\tau) v'(t(\tau))$ and plugging those facts into~\eqref{eq:ode-sys-NAG} gives the scaled NAG flow
\begin{equation}\label{eq:ode-sys-NAG-rescale}
	\left\{
	\begin{aligned}
		{}& 			y' =\alpha (w  - y ),\\
		{}&			\beta w' = \mu\alpha  
		\big(y - w \big) - 
		\alpha \nabla f(y ),
	\end{aligned}
	\right.
\end{equation}
with initial conditions $y(0)=x_0$ and $y'(0)=\alpha(0)x'(0)$. By Remark \ref{rem:gamma}, the equation~\eqref{eq:gammat-NAG} can be replaced by $\gamma'\leqslant \mu-\gamma$, which becomes
\begin{equation}
	\label{eq:gammat-NAG-rescale}
	\beta'\leqslant
	\alpha(\mu-\beta),
	\quad \beta(0)=\gamma_0.
\end{equation}
Correspondingly, the Lyapunov function~\eqref{eq:Lt-NAG} reads as follows
\begin{equation*}
	\widetilde{\mathcal L}(\tau):=
	f(y(\tau))-f(x^*)+\frac{\beta(\tau)}{2}
	\nm{w(\tau)-x^*}^2,\quad \tau\geqslant 0.
\end{equation*}
Analogously to~\eqref{eq:dLt}, we can prove 
\[
\widetilde{\mathcal L}'\leqslant
-\alpha\widetilde{\mathcal L}
- \frac{\mu\alpha}{2}
\|w-y\|^{2},
\]
which implies that
\begin{equation}\label{eq:conv-L-rescale}
	\widetilde{\mathcal L}(\tau)
	\leqslant 
	{\mathrm e}^{-\int_0^\tau\alpha (s)\,{\rm d} s}
	\widetilde{\mathcal L}(0),
	\quad \tau\geqslant 0.
\end{equation}	
Therefore, larger scaling factor $\alpha$ promises faster decay rate.

We note that the scaled NAG flow~\eqref{eq:ode-sys-NAG-rescale} 
is very close to the two models~\eqref{eq:ode-jordan} and 
\eqref{eq:Wilson-ode}, which are derived 
in~\cite{Wibisono;Wilson;Jordan:2016variational} and \cite{Wilson:2018} 
respectively, via the variational perspective. 
Indeed, they differs mainly from the coefficient of $w'$.
By~\eqref{eq:gammat-NAG-rescale}, 
an elementary calculation gives 
\[
\beta(\tau) \leqslant
\mu + (\gamma_0-\mu)
{\mathrm e}^{-\int_0^\tau
	\alpha (s)\,{\rm d} s},
\quad \tau\geqslant 0.
\]
Therefore,~\eqref{eq:ode-sys-NAG-rescale} chooses variable
coefficient $\beta(\tau)$ for $\mu\geqslant 0$, 
while~\eqref{eq:ode-jordan} considers dynamically changing 
coefficient~\eqref{eq:beta-jordan} only for $\mu=0$ 
and~\eqref{eq:Wilson-ode} adopts 
fixed parameter $\mu>0$. For strongly convex case $\mu>0$,
if we take $\beta =\mu$, which satisfies 
\eqref{eq:gammat-NAG-rescale}, then the scaled system
\eqref{eq:ode-sys-NAG-rescale} coincides with 
~\eqref{eq:Wilson-ode}. For convex case $\mu=0$, if both~\eqref{eq:ode-jordan} 
and~\eqref{eq:gammat-NAG-rescale} are equalities, then 
\eqref{eq:ode-sys-NAG-rescale} agrees with~\eqref{eq:ode-jordan}. 
Hence, we conclude that our NAG flow system is more tight and 
provides a unified way to handle $\mu=0$ and $\mu>0$.

Now, let us look at a concrete rescaling example. 
Let the scaling factor $\alpha$ solve
\begin{equation}
	\label{eq:a}
	2\alpha' \leqslant \mu-\alpha^2,\quad\alpha(0)=\sqrt{\gamma_0}.
\end{equation}
For instance, the following choice is allowed:
\begin{equation}\label{eq:alp-1}
	\alpha(\tau)=\frac{\sqrt{\gamma_0}\ b}
	{\sqrt{\gamma_0}\, \tau+b},\quad0<b\leqslant2.
\end{equation}
For the equality case of~\eqref{eq:a}, 
we have a closed-form solution 
\begin{equation}\label{eq:alp-2}
	\alpha(\tau) = 
	\left\{
	\begin{aligned}
		&\frac{2\sqrt{\gamma_0}}
		{\sqrt{\gamma_0}\, \tau+2},&&\text{ if }\mu=0,\\
		&\sqrt{\mu}\cdot
		\frac{e^{\sqrt{\mu} \, \tau}-\alpha_\mu}
		{e^{\sqrt{\mu} \, \tau}+\alpha_\mu},&&\text{ if }\mu>0,
	\end{aligned}
	\right.
\end{equation}
where 
\[
\alpha_\mu=
\frac{\sqrt{\mu}-\sqrt{\gamma_0}}
{\sqrt{\mu}+\sqrt{\gamma_0}}\in(-1,1).
\]
We now set $\beta = \alpha^2$ 
which fulfills~\eqref{eq:gammat-NAG-rescale} by our assumption~\eqref{eq:a},
then the scaled NAG flow~\eqref{eq:ode-sys-NAG-rescale}
gives a new HB system
\begin{equation}\label{eq:y}
	y'' + \frac{1}{\alpha}
	\left(\mu+\alpha^2-\alpha'\right)y'+\nabla f(y)=0.
\end{equation}
According to~\eqref{eq:conv-L-rescale}, we have the estimate
\[
\widetilde{\mathcal L}(\tau)
\leqslant 
\left\{
\begin{aligned}
	{}&\frac{b^b\widetilde{\mathcal L}(0)}{(\sqrt{\gamma_0}\tau+b)^b},
	&&\text{ if } \alpha \text{ satisfies~\eqref{eq:alp-1}},\\
	{}&	\frac{(1+\alpha_\mu)^2\widetilde{\mathcal L}(0)}{
		\left(
		e^{\sqrt{\mu} \tau/2}+\alpha_\mu e^{-\sqrt{\mu} \tau/2}
		\right)^2},&&\text{ if } \alpha \text{ satisfies~\eqref{eq:alp-2} and } \mu>0.
\end{aligned}
\right.
\]
Particularly, if $\mu>0$ and $\alpha$ satisfies~\eqref{eq:alp-2} with $\gamma_0=\sqrt{\mu}$, then $\alpha(\tau)=\sqrt{\mu}$ and~\eqref{eq:y} recovers~\eqref{eq:Wilson-ode-par} with the same rate $O(e^{-\sqrt{\mu}\tau})$. Moreover, if $\mu=0$ and $\alpha$ satisfies~\eqref{eq:alp-1} with $\gamma_0=4$ and $b=2$, then $\alpha(\tau)=2/(\tau+1)$ and~\eqref{eq:y} becomes 
\[
y'' + \frac{3}{\tau+1}y'+\nabla f(y)=0,\quad\tau>0,
\]
which gives the decay rate $O(\tau^{-2})$ and coincides with the prevailing ODE model \eqref{eq:Su2015-ode} derived in~\cite{Su;Boyd;Candes:2016differential}.
\section{An Implicit Scheme}
\label{sec:im}
Exponential decay of an implicit discretization for solving
\eqref{eq:ode-sys-NAG} can be established, which is more or 
less straightforward since one can easily follow the proof from the continuous problem. 
However, the implicit scheme requires efficient solver or proximal calculation and may not be practical sometimes. It is presented here to bridge the analysis from the continuous level to semi-implicit and explicit schemes.  

Consider the following implicit scheme 
\begin{equation}\label{eq:im-ode-NAG}
	\left\{
	\begin{aligned}
		\frac{x_{k+1}-x_{k}}{\alpha_k}={}& v_{k+1}-x_{k+1},\\
		\frac{v_{k+1}-v_{k}}{\alpha_k}={}&
		\frac{\mu}{\gamma_k}  (x_{k+1}-v_{k+1})
		-\frac{1}{\gamma_k}\nabla f(x_{k+1}),
	\end{aligned}
	\right.
\end{equation}
where $\alpha_k>0$ denotes the time step size to discretize the time derivative and the 
parameter equation~\eqref{eq:gammat-NAG} 
is also discretized implicitly
\begin{equation}\label{eq:im-gammat-NAG}
	\frac{		\gamma_{k+1}  - \gamma_{k}}{\alpha_k}  
	=\mu-\gamma_{k+1},
	\quad \gamma_0>0.
\end{equation}

We shall present the convergence result 
for the implicit scheme~\eqref{eq:im-ode-NAG}. 
To do so, we introduce a suitable Lyapunov function
\begin{equation}
	\label{eq:Lk}
	\mathcal L_k:={}
	f(x_k)-f(x^*) +
	\frac{\gamma_k}{2}
	\nm{v_k-x^*}^2,
\end{equation}
which is clearly a discrete analogue to the continuous 
one~\eqref{eq:Lt-NAG}. 
\begin{theorem}\label{thm:conv-im-ode-NAG}
	If $ f\in\mathcal S_{\mu}^{1}$ with $\mu\geqslant0$, then 
	for the scheme~\eqref{eq:im-ode-NAG} with $\alpha_k>0$, we have
	\begin{equation*}
		\mathcal L_{k+1}\leqslant 
		\frac{	\mathcal L_k}{1+\alpha_k },
		\quad k\in\mathbb N.
	\end{equation*}
\end{theorem}
\begin{proof}
	It suffices to prove
	\begin{equation}\label{eq:diff-Lk-NAG-im}
		\mathcal L_{k+1}-\mathcal L_{k}\leqslant 
		-\alpha_k 
		\mathcal L_{k+1}.
	\end{equation}
	
	Let us mimic the proof of Lemma \ref{lem:conv-ode-NAG}. 
	Instead of the derivative, we compute the difference as follows
	\[
	\begin{split}
		\mathcal L_{k+1}-\mathcal L_{k} 
		={}&f(x_{k+1})-f(x_k) + 
		\frac{\gamma_{k+1}-\gamma_k}{2}
		\nm{v_{k+1}-x^*}^2\\
		{}&\quad+\frac{\gamma_k}{2}
		\left(\nm{v_{k+1}-x^*}^2-\nm{v_{k}-x^*}^2 
		\right)\\
		={}&f(x_{k+1})-f(x_k) + 
		\frac{\alpha_k}{2}(\mu- \gamma_{k+1})
		\nm{v_{k+1}-x^*}^2\\
		{}&\quad+\gamma_k
		\inner{v_{k+1}-v_k, (v_{k+1}+v_{k})/2 -x^*}.
	\end{split}
	\]
	Analogously to the continuous level, we focus on the last term
	\[
	\begin{aligned}
		{}&\gamma_k
		\inner{v_{k+1}-v_k, (v_{k+1}+v_{k})/2 -x^*} \\
		= {}&
		\gamma_k
		\inner{v_{k+1}-v_k, v_{k+1} -x^*} - 
		\frac{\gamma_k}2
		\nm{v_{k+1}-v_k}^2.
	\end{aligned}
	\]
	By~\eqref{eq:im-ode-NAG}, it follows that
	\[
	\begin{split}
		{}&\gamma_k\inner{v_{k+1}-v_k, v_{k+1} -x^*} \\
		= {}&
		\mu\alpha_k\inner{x_{k+1} - v_{k+1}, v_{k+1} - x^{*}}-
		\alpha_k \dual{ \nabla f(x_{k+1}), v_{k+1} - x^{*}},
	\end{split}
	\]
	and  we use Lemma \ref{lem:cross} to split the cross term into squares:
	\begin{equation*}
		\begin{split}
			{}&2\inner{x_{k+1} - v_{k+1}, v_{k+1} - x^{*}} \\
			={}&\left\|x_{k+1}-x^{*}\right\|^{2}-\|x_{k+1}-v_{k+1}\|^{2}
			-\left\|v_{k+1}-x^{*}\right\|^{2}.
		\end{split}
	\end{equation*}
	For the gradient term, we have $v_{k+1}-x^{*} = v_{k+1}-x_{k+1} + x_{k+1}  - x^{*}$ and use~\eqref{eq:im-ode-NAG} to obtain
	\begin{align*}
		{}&-\alpha_k \dual{ \nabla f(x_{k+1}), v_{k+1} - x^{*}} \\
		={}& - \dual{ \nabla f(x_{k+1}), x_{k+1} - x_{k}} 
		- \alpha_k\dual{   \nabla f(x_{k+1}), x_{k+1} - x^{*}}.
	\end{align*}
	Consequently, using the $\mu$-strongly convex 
	property (cf.\eqref{eq:convex-mu}) of $f$ and dropping surplus negative square terms, we see
	\[
	\begin{split}
		\mathcal L_{k+1}-\mathcal L_{k} 
		\leqslant{}&-\alpha_k \mathcal L_{k+1}.
	\end{split}
	\]
	This proves~\eqref{eq:diff-Lk-NAG-im} 
	and concludes the proof of this {theorem}.
	
\end{proof}

We observe from Theorem \ref{thm:conv-im-ode-NAG}  that the fully implicit scheme
\eqref{eq:im-ode-NAG} achieves linear convergence rate as long as 
$\alpha_k\geqslant\alpha>0$ for all $k>0$ and 
larger $\alpha_k$ yields faster convergence rate. 
We also mention that~\eqref{eq:im-ode-NAG} can be rewritten as 
\begin{equation}\label{eq:prox-im}
	\left\{
	\begin{split}
		x_{k+1} ={}& \prox_{\eta_k f}(y_k),\\
		v_{k+1}={}&x_{k+1}+	\frac{x_{k+1}-x_k}{\alpha_k},
	\end{split}
	\right.
\end{equation}
where the proximal operator $\prox_{\eta_k f}$ has been introduced in \eqref{eq:prox} and
\[
\gamma_{k+1} = {}\frac{\gamma_k + \mu \alpha_k}{1+\alpha_k}, \,
\eta_k = {}\frac{\alpha_k^2}{\gamma_k+(\mu+\gamma_k)\alpha_k},
\,
y_k ={} \frac{\gamma_k\alpha_kv_k+(\gamma_k+\mu\alpha_k)x_k}
{\gamma_k+(\mu+\gamma_k)\alpha_k}.
\]
Therefore, it allows $f$ to be nonsmooth and we claim that 
Theorem \ref{thm:conv-im-ode-NAG} still holds true in this case. 
One just replaces the gradient $\nabla f(x_{k+1})$ with the subgradient $(y_k-x_{k+1})/\eta_k\in\partial f(x_{k+1})$; see~\eqref{eq:def-subg} and~\eqref{eq:prox-subg}.

For convex case, i.e., $\mu=0$, our method~\eqref{eq:prox-im} is very close to G\"{u}ler's proximal point algorithm~\cite{guler_new_1992}
\[
\left\{
\begin{split}
	x_{k+1} ={}& \prox_{\eta_k f}(y_k),\quad \eta_k =\alpha_k^2/\gamma_{k+1},\\
	v_{k+1}={}&x_{k}+	\frac{x_{k+1}-x_k}{\alpha_k},
\end{split}
\right.
\]
where $\gamma_{k+1}-\gamma_{k}=-\alpha_{k}\gamma_k$ and $ y_k ={}  \alpha_kv_k+(1-\alpha_k)x_k$.
Indeed, with suitable step size, they share the similar rate; see~\cite[Theorem 2.3]{guler_new_1992} 
and Theorem \ref{thm:appa} below.
\begin{theorem}\label{thm:appa}
	If $f$ is proper, closed and convex and we choose $\alpha_k^2=\eta_k\gamma_k(1+\alpha_k)$ with $\eta_k>0$, then for the 
	proximal point algorithm~\eqref{eq:prox-im} with $\mu=0$, we have
	\begin{equation}\label{eq:rate-Lk-appa}
		\frac{\mathcal L_{0}}{(1+\sum_{i=0}^{k-1}\sqrt{\gamma_0\eta_i})^2}
		\leqslant \mathcal L_{k}\leqslant {}	
		\frac{4\mathcal L_{0}}{(2+\sum_{i=0}^{k-1}\sqrt{\gamma_0\eta_i})^2},
	\end{equation}
	which means if $\sum_{k=0}^{\infty}\sqrt{\eta_k}=\infty$ then $\mathcal L_k\to0$ as $k\to\infty$.
	Moreover, it holds that
	\begin{equation}
		\label{eq:rate-Lk-appa-L}
		\mathcal L_{k}\leqslant {}
		\frac{4}{\sum_{i=0}^{k-1}\sqrt{\eta_i}}
		\left(
		\frac{1}{\gamma_0}\left(f(x_0)-f(x^*)\right)
		+\frac{1}{2}\nm{v_0-x^*}^2
		\right).
	\end{equation}
\end{theorem}
\begin{proof}
	For convenience and later use, define a sequence $\{\rho_{k}\}$ by that
	\begin{equation}\label{eq:rhok}
		\rho_0=1,\quad\rho_k:=		\prod_{i=0}^{k-1}\frac{1}{1+\alpha_i},\quad k\geqslant 1.
	\end{equation}	
	As mentioned above, Theorem \ref{thm:conv-im-ode-NAG} holds true for such a nonsmooth $f$ and thus it is evident that $\mathcal L_k\leqslant \rho_k\mathcal L_0$. Invoking Lemma \ref{lem:gk-ak-im} proves~\eqref{eq:rate-Lk-appa} 
	and it is trivial to obtain~\eqref{eq:rate-Lk-appa-L} from~\eqref{eq:rate-Lk-appa}.
	This finishes the proof.
		\end{proof}
\begin{remark}
	\label{rem:gk-bd}
	Note that the sequence $\{\gamma_k\}$ in~\eqref{eq:im-gammat-NAG} is 
	bounded: $0<\gamma_{k}\leqslant \max\{\mu,\gamma_{0}\}$ and $\gamma_k\to\mu$ as $k\to\infty$. Hence, even for large $\gamma_0$, the Lyapunov function $\mathcal L_k$ is asymptotically bounded as $k\to\infty$. In addition, from~\eqref{eq:rate-Lk-appa} and~\eqref{eq:rate-Lk-appa-L}, we see that, for small $\gamma_0$, the convergence rate  depends on $\gamma_0$ but large $\gamma_0$ does not pollute the final rate. This fact also holds true for all the forthcoming convergence bounds.	
\end{remark}
\section{Gauss-Seidel Splitting with Corrections}
\label{sec:GS-correc}
This section considers the Gauss-Seidel splitting~\eqref{eq:GS}, which is a semi-implicit discretization.
In Section \ref{sec:GS}, we have established the spectrum bound $O(1-\sqrt{\mu/L})$ with step size $\alpha_k=O(\sqrt{\mu/L})$ for quadratic objectives. However, as we summarized in Section \ref{sec:limit}, spectrum analysis is not sufficient for (norm) convergence.

Indeed, in the sequel, we further show that, for the discrete Lyapunov function~\eqref{eq:Lk}, with any step size $\alpha_k>0$, the naive discretization~\eqref{eq:GS}, reformulated as~\eqref{eq:ex1-ode-NAG}, does not lead to the contraction property like~\eqref{eq:diff-Lk-NAG-im}. Therefore,  this motivates us to add some proper correction steps.
\subsection{The Gauss-Seidel splitting}
Recall the Gauss-Seidel splitting~\eqref{eq:GS}: 
given step size $\alpha_k>0$ and previous result
$(x_k,v_k)$, compute $(x_{k+1},v_{k+1})$ from
\begin{equation}\label{eq:ex1-ode-NAG}
	\left\{
	\begin{aligned}
		\frac{x_{k+1}-x_{k}}{\alpha_k}={}& v_{k}-x_{k+1},\\
		\frac{v_{k+1}-v_{k}}{\alpha_k}={}&
		\frac{\mu}{\gamma_k}(x_{k+1}-v_{k+1})
		-\frac{1}{\gamma_k}\nabla f(x_{k+1}).
	\end{aligned}
	\right.
\end{equation}
In addition, the parameter equation~\eqref{eq:gammat-NAG} of $\gamma$ 
is still discretized implicitly via~\eqref{eq:im-gammat-NAG}.
\begin{lemma}
	\label{lem:bd-NAG-GS-1st}
	If $f\in\mathcal {S}_{\mu}^{1}$ with $\mu\geqslant 0$, then 
	for~\eqref{eq:ex1-ode-NAG} with any step size $\alpha_k>0$, we have
	\begin{equation}\label{eq:bd-NAG-GS-1st}
		\mathcal L_{k+1}-\mathcal L_{k} 
		\leqslant-\alpha_k \mathcal L_{k+1}
		-	\frac{\gamma_k}2\nm{v_{k+1}-v_k}^2
		-\alpha_k\dual{\nabla f(x_{k+1}),v_{k+1} - v_k},
	\end{equation}
	and
	\begin{equation}\label{eq:bd-NAG-GS-1st-gd}
		\mathcal L_{k+1}-\mathcal L_{k} 
		\leqslant-\alpha_k \mathcal L_{k+1}
		+	\frac{\alpha_k^2}{2\gamma_k}
		\nm{\nabla f(x_{k+1})}_{*}^2.
	\end{equation}
\end{lemma}
\begin{proof}
	Following the proof of Theorem \ref{thm:conv-im-ode-NAG},
	we start from the difference
	\[
	\begin{split}
		\mathcal L_{k+1}-\mathcal L_{k}
		={}& f(x_{k+1})-f(x_k)-
		\frac{\alpha_k \gamma_{k+1}}{2}
		\nm{v_{k+1}-x^*}^2\\
		{}&\quad	-\frac{\mu\alpha_k}{2}
		\nm{x_{k+1}-v_{k+1}}^2-
		\frac{\gamma_k}2
		\nm{v_{k+1}-v_k}^2\\
		{}&\quad	+\frac{\mu\alpha_k}{2}
		\nm{x_{k+1}-x^*}^2
		-\alpha_k \dual{ \nabla f(x_{k+1}), v_{k+1} - x^{*}}.	
	\end{split}
	\]
	Using the update for $x_{k+1}$ in~\eqref{eq:ex1-ode-NAG}, 
	we split the gradient term as below
	\begin{align*}
		{}&-\alpha_k \dual{ \nabla f(x_{k+1}), v_{k+1} - x^{*}} \\
		={}& -\alpha_k\dual{\nabla f(x_{k+1}),v_{k+1} - v_k} -
		\dual{\nabla f(x_{k+1}),\alpha_k(v_k-x_{k+1})}\\
		{}&\quad -\alpha_k\dual{  \nabla f(x_{k+1}), x_{k+1} - x^{*}}\\
		={}&-\alpha_k\dual{\nabla f(x_{k+1}),v_{k+1} - v_k}
		- \dual{ \nabla f(x_{k+1}), x_{k+1} - x_{k}} \\
		{}&\quad	- \alpha_k\dual{  \nabla f(x_{k+1}), x_{k+1} - x^{*}}.
	\end{align*}
	As $f\in\mathcal {S}_{\mu}^{1}$, 
	we obtain that
	\[
	\begin{split}
		\mathcal L_{k+1}-\mathcal L_{k} 
		\leqslant{}&-\alpha_k \mathcal L_{k+1}
		-	\frac{\gamma_k}2\nm{v_{k+1}-v_k}^2
		-\alpha_k\dual{\nabla f(x_{k+1}),v_{k+1} - v_k}\\
		{}&\quad-\frac{\mu\alpha_k}{2}\nm{x_{k+1}-v_{k+1}}^2
		-\frac{\mu}{2}\nm{x_{k+1}-x_k}^2.
	\end{split}
	\]
	Ignoring all the negative terms of the second line, the above estimate implies~\eqref{eq:bd-NAG-GS-1st}.
	
	As we see, different from~\eqref{eq:diff-Lk-NAG-im}, the estimate 
	\eqref{eq:bd-NAG-GS-1st} contains a combination of a negative term and another cross term. Obviously, an easy application of Cauchy-Schwarz inequality yields
	\[
	-	\frac{\gamma_k}2\nm{v_{k+1}-v_k}^2
	-\alpha_k\dual{\nabla f(x_{k+1}),v_{k+1} - v_k}\leqslant \frac{\alpha_k^2}{2\gamma_k}\nm{\nabla f(x_{k+1})}_{*}^2.
	\]
	This yields another bound~\eqref{eq:bd-NAG-GS-1st-gd} that only involves a positive gradient norm.  	
\end{proof}
\subsection{A predictor-corrector method}
\label{sec:NAG-GS-1st-extra}
To conquer the cross 
term $	-\alpha_k\dual{\nabla f(x_{k+1}),v_{k+1} - v_k}$ in 
\eqref{eq:bd-NAG-GS-1st}, we add an extra extrapolation step 
to~\eqref{eq:ex1-ode-NAG} which can be thought as an semi-implicit discretization of $x'= v - x$ with the newest update $v_{k+1}$. More precisely, consider
\begin{equation}\label{eq:NAG-GS-1st-extra}
	\left\{
	\begin{aligned}
		\frac{y_{k}-x_{k}}{\alpha_k}={}& v_{k}-y_{k},\\
		\frac{v_{k+1}-v_{k}}{\alpha_k}={}&
		\frac{\mu}{\gamma_k}(y_{k}-v_{k+1})
		-\frac{1}{\gamma_k}\nabla f(y_{k}),\\
		\frac{x_{k+1}-x_{k}}{\alpha_k}={}& v_{k+1}-x_{k+1}.
	\end{aligned}
	\right.
\end{equation}
This is in line with the spirit of the predictor-corrector method for ODE solvers \cite[Section 3.8]{Suli:2010Numerical}. The variable $y_k$ is the predictor produced by an explicit scheme and $x_{k+1}$ is the corrector by an implicit scheme. It can be also thought of as a symmetric Gauss-Seidel iteration for approximating the implicit Euler method. Again, the parameter equation~\eqref{eq:gammat-NAG} of $\gamma$ 
is still discretized via~\eqref{eq:im-gammat-NAG}.

As the first two steps of \eqref{eq:NAG-GS-1st-extra} agree with \eqref{eq:ex1-ode-NAG}, with $x_{k+1}$ being $y_k$, recalling the estimate \eqref{eq:bd-NAG-GS-1st}, we have
\[
\widehat{	\mathcal L}_{k}  -\mathcal L_{k}
\leqslant 	-\alpha_k	
\widehat{\mathcal L}_{k}
-	\frac{\gamma_k}2\nm{v_{k+1}-v_k}^2
-\alpha_k\dual{\nabla f(y_{k}),v_{k+1} - v_k},
\]
where 
\begin{equation}\label{eq:hat-Lk}
	\widehat{		\mathcal L}_k := f(y_k)-f(x^*)+\frac{\gamma_{k+1}}2\nm{v_{k+1}-x^*}^2.
\end{equation}
Therefore, it follows that
\[
\widehat{	\mathcal L}_{k} 
\leqslant 	
\frac{\mathcal L_k}{1+\alpha_k}
-	\frac{\gamma_k}{2(1+\alpha_k)}\nm{v_{k+1}-v_k}^2
-\frac{\alpha_k}{1+\alpha_k}\dual{\nabla f(y_{k}),v_{k+1} - v_k}.
\]

From the update for $y_k$ and $x_{k+1}$ in~\eqref{eq:NAG-GS-1st-extra}, we find the relation
\[
x_{k+1}-y_k = \frac{\alpha_k}{1+\alpha_k}(v_{k+1}-v_k),
\]
and if $f\in\mathcal S^{1,1}_{\mu,L}$, then there comes the estimate (cf. \eqref{eq:Lip-L-equiv})
\[
\begin{split}
	{}&\mathcal L_{k+1}-	\widehat{		\mathcal L}_k  = f(x_{k+1})-f(y_k)\\
	\leqslant {}&\dual{\nabla f(y_k),x_{k+1}-y_k}+\frac{L}{2}\nm{x_{k+1}-y_k}^2\\
	={}&\frac{\alpha_k}{1+\alpha_k}\dual{\nabla f(y_{k}),v_{k+1}-v_k}
	+\frac{L\alpha_k^2}{2(1+\alpha_k)^2}\nm{v_{k+1}-v_k}^2.
\end{split}
\]
As a result, we obtain
\begin{equation}\label{eq:diff-1}
	\mathcal L_{k+1}\leqslant \frac{\mathcal L_k}{1+\alpha_k}
	+\left(
	\frac{L\alpha_k^2}{2(1+\alpha_k)^2}-\frac{\gamma_k}{2(1+\alpha_k)}
	\right)\nm{v_{k+1}-v_k}^2.
\end{equation}
The second term vanishes if we choose suitable step size; see the theorem below.
\begin{theorem}
	\label{thm:conv-ex0-ode-NAG}
	Assume that $f\in\mathcal {S}_{\mu,L}^{1,1}$ 
	with $0\leqslant \mu\leqslant L<\infty$
	and $L\alpha_k^2
	=\gamma_k(1+\alpha_k)$, then for the predictor-corrector 
	scheme~\eqref{eq:NAG-GS-1st-extra} together 
	with~\eqref{eq:im-gammat-NAG}, we have
	\begin{equation}\label{eq:decay-Lk-ex0}
		\mathcal L_{k+1}
		\leqslant 
		\frac{	 \mathcal L_k}{1+\alpha_k },		\quad k\in\mathbb N,
	\end{equation}
	where $\mathcal L_k$ is defined by~\eqref{eq:Lk}.
	Consequently, for all $k\geqslant 0$,
	\begin{equation}\label{eq:rate-Lk-ex0}
		\mathcal L_{k}\leqslant {}	\mathcal L_{0}\times
		\min\left\{
		\frac{4L}{(\sqrt{\gamma_0}\, k+2\sqrt{L})^2},\,
		\left(1+\sqrt{\frac{\min\{\gamma_0,\mu\}}L}\right)^{-k}
		\right\},
	\end{equation}
	and moreover, 
	for all $k\geqslant 1$,
	\begin{equation}
		\label{eq:rate-Lk-ex0-L}
		\begin{split}
			\mathcal L_{k}\leqslant {}&
			C_{\gamma_0,L}\times
			\min\left\{
			\frac{4}{k^2},\,
			\left(1+\sqrt{\frac{\min\{\gamma_0,\mu\}}L}\right)^{1-k}
			\right\},
		\end{split}
	\end{equation}
	where 
	\begin{equation}\label{eq:CgL}
		C_{\gamma_0,L}:=	\frac{L}{\gamma_0}
		\big(	f(x_0)-f(x^*)\big)+\frac{L}{2}
		\nm{v_0-x^*}^2.
	\end{equation}
\end{theorem}
\begin{proof}
	The inequality~\eqref{eq:diff-1} suggests the choice  $L\alpha_k^2
	=\gamma_k(1+\alpha_k)$ and promises~\eqref{eq:decay-Lk-ex0}.
	Recalling the sequence $\{\rho_{k}\}$ defined by~\eqref{eq:rhok}, we have $	\mathcal L_k\leqslant \rho_k\mathcal L_0$.
	Hence, using Lemma \ref{lem:gk-ak-im} gives the decay estimate 
	of $\rho_k$ and proves~\eqref{eq:rate-Lk-ex0}.
	
	It remains to check~\eqref{eq:rate-Lk-ex0-L} for all $k\geqslant 1$. 
	From Lemma \ref{lem:gk-ak-im} we easily get
	\begin{equation}\label{eq:ekL0-1}
		\begin{aligned}
			\rho_k\mathcal L_0
			\leqslant  &
			\left(
			f(x_0)-f(x^*)+\frac{\gamma_0}{2}
			\nm{v_0-x^*}^2
			\right)\times 	\frac{4L}{(\sqrt{\gamma_0}\, k+2\sqrt{L})^2}
			\leqslant  {}\frac{4C_{\gamma_{0},L}}{k^2}\\
		\end{aligned}
	\end{equation} 
	On the other hand, by the relation $L\alpha_0^2=\gamma_0(1+\alpha_0)$, 
	it is evident that
	\[
	\alpha_0 = \frac{1}{2L}
	\left(\gamma_0+\sqrt{4\gamma_0L+\gamma_0^2}\right),
	\]
	which implies 
	\[
	\frac{1}{1+\alpha_0}=
	\frac{2L}
	{\gamma_0+2L+\sqrt{4\gamma_0L+\gamma_0^2}}
	\leqslant \frac{L}{\gamma_0}.
	\]
	The above estimate also indicates that
	\begin{equation*}
		\rho_k\mathcal L_0
		= \frac{\mathcal L_0}{1+\alpha_0}
		\frac{\rho_k}{\rho_1}
		\leqslant
		C_{\gamma_{0},L}\frac{\rho_k}{\rho_1}
		=C_{\gamma_{0},L}
		\times 
		\prod_{i=1}^{k-1}\frac{1}{1+\alpha_i}.
	\end{equation*}
	Applying Lemma \ref{lem:gk-ak-im} shows that $\alpha_{k}\geqslant \sqrt{\min\{\gamma_0,\mu\}/L}$ and it follows that
	\[
	\rho_k\mathcal L_0
	\leqslant 
	C_{\gamma_{0},L}
	\times
	\left(1+\sqrt{\min\{\gamma_0,\mu\}/L}\right)^{1-k}.
	\]
	Collecting this estimate and~\eqref{eq:ekL0-1} establishes 
	the final rate~\eqref{eq:rate-Lk-ex0-L} and thus 
	completes the proof of this theorem.
	
\end{proof}
\begin{remark}
	We mention that the estimate \eqref{eq:rate-Lk-ex0-L} verifies the claim made previously in Remark \ref{rem:gk-bd}. That is, the convergence rate given in Theorem \ref{thm:conv-ex0-ode-NAG} depends on small $\gamma_0$ but is
	robust when $\gamma_0\geqslant L$.
\end{remark}
\subsection{Correction via a gradient step}
Motivated by the estimate~\eqref{eq:bd-NAG-GS-1st-gd}, we can also aim to 
cancel the gradient norm square. One preferable 
choice is the gradient descent step and according to our discussion below, any other correction step satisfying the decay property~\eqref{eq:decay2} is acceptable. Note that the two numerical schemes proposed in~\cite{Siegel:2019} and~\cite{Wilson:2018} for the HB equation~\eqref{eq:Wilson-ode-par} also have additional gradient steps. 

As what we did before, replace $x_{k+1}$ by $y_k$ in 
\eqref{eq:ex1-ode-NAG} and consider the following corrected scheme:
given $\alpha_k>0$ and 
$(x_k,v_k)$, compute $(x_{k+1},v_{k+1})$ from
\begin{equation}\label{eq:ex-1-NAG}
	\left\{
	\begin{aligned}
		\frac{y_{k}-x_{k}}{\alpha_k}={}& v_{k}-y_{k},\\
		\frac{v_{k+1}-v_{k}}{\alpha_k}={}&
		\frac{\mu}{\gamma_k}(y_{k}-v_{k+1})
		-\frac{1}{\gamma_k}\nabla f(y_{k}),\\
		x_{k+1}-y_k={}&-\frac{1}{L}\nabla f(y_k).
	\end{aligned}
	\right.
\end{equation}
The implicit discretization~\eqref{eq:im-gammat-NAG} for the parameter equation~\eqref{eq:gammat-NAG} keeps unchanged here.
In the first equation $y_k$ can be solved in terms of the known data $(x_k, v_k)$. After that, we evaluate the gradient $\nabla f(y_k)$ once and use it to update $(x_{k+1}, v_{k+1})$. 
\begin{theorem}
	\label{thm:conv-ex1-ode-NAG}
	Assume that $f\in\mathcal {S}_{\mu,L}^{1,1}$ 
	with $0\leqslant \mu\leqslant L<\infty$ and $L\alpha_k^2
	=\gamma_k(1+\alpha_k)$, then for the corrected 
	scheme~\eqref{eq:ex-1-NAG} together with~\eqref{eq:im-gammat-NAG}, we have
	\begin{equation}\label{eq:decay-Lk-ex1}
		\mathcal L_{k+1}
		\leqslant 
		\frac{	 \mathcal L_k}{1+\alpha_k },		\quad k\in\mathbb N,
	\end{equation}
	where $\mathcal L_k$ is defined by~\eqref{eq:Lk}, and both the 
	two estimates~\eqref{eq:rate-Lk-ex0} and 
	\eqref{eq:rate-Lk-ex0-L} hold true here.
\end{theorem}
\begin{proof}
	According to~\eqref{eq:bd-NAG-GS-1st-gd} in Lemma \ref{lem:bd-NAG-GS-1st}, we have established that
	\begin{equation}\label{eq:mid-bd-Lk}
		\widehat{\mathcal L}_{k}-\mathcal L_{k} 
		\leqslant-\alpha_k \widehat{\mathcal L}_{k}+
		\frac{\alpha_k^2}{2\gamma_k}
		\nm{\nabla f(y_{k})}_{*}^2,
	\end{equation}
	where $\widehat{\mathcal L}_k$ is defined by~\eqref{eq:hat-Lk}.
	Thanks to the additional gradient step 
	in~\eqref{eq:ex-1-NAG}, we have the basic gradient descent inequality:
	\begin{equation}
		\label{eq:decay2}
		f(x_{k+1})-f(y_k)\leqslant 
		-\frac{1}{2L}
		\nm{\nabla f(y_k)}_{*}^2,
	\end{equation}
	which comes from \eqref{eq:Lip-L-equiv} since $f\in\mathcal S_{\mu,L}^{1,1}$ and implies that 
	\[
	\mathcal L_{k+1}\leqslant
	\widehat{\mathcal L}_k
	-\frac{1}{2L}
	\nm{\nabla f(y_k)}_{*}^2.
	\]
	Plugging this into~\eqref{eq:mid-bd-Lk} gives 
	\[
	\mathcal L_{k+1}-\mathcal L_{k} 
	\leqslant-\alpha_k \mathcal L_{k+1}
	+\frac{1}{2L\gamma_k}
	\left( L\alpha_k^2 -\gamma_k(1+\alpha_k)\right)
	\nm{\nabla f(y_k)}_{*}^2.
	\]
	This together with the condition $L\alpha_k^2
	=\gamma_k(1+\alpha_k)$ yields~\eqref{eq:decay-Lk-ex1}. 
	
	As we choose the same step size as Theorem \ref{thm:conv-ex0-ode-NAG}, based on the contraction~\eqref{eq:decay-Lk-ex1}, it is trivial to conclude that the two estimates~\eqref{eq:rate-Lk-ex0} and~\eqref{eq:rate-Lk-ex0-L} hold true here indeed. This completes the proof of this theorem.
	 \end{proof}	
\section{A Corrected Semi-implicit Scheme from NAG Method}
\label{sec:NAG}
In this section, we consider another semi-implicit scheme 
which comes exactly from Nesterov accelerated gradient method.
\subsection{NAG method}
In~\cite[Chapter 2, General scheme of optimal method]{Nesterov:2013Introductory}, 
by using the estimate sequence, Nesterov presented 
an accelerated gradient method for solving~\eqref{min} 
with 
$f\in\mathcal {S}_{\mu,L}^{1,1}$ 
with $0\leqslant \mu\leqslant L<\infty$;
see Algorithm \ref{algo:NAG} below.
\begin{algorithm}[H]
	\caption{Nesterov Accelerated Gradient (NAG) Method}
	\label{algo:NAG}
	\begin{algorithmic}[1] 
		\REQUIRE  $x_0,v_0\in V$ and $\gamma_0>0$.
		\FOR{$k=0,1,\ldots$}
		\STATE Compute $\alpha_k\in(0,1)$ from $		L\alpha_k^2 = (1-\alpha_k)\gamma_k+\mu\alpha_k$.
		\STATE Update $\displaystyle{\gamma_{k+1} = (1-\alpha_k)\gamma_k+\mu\alpha_k}$.
		\STATE Set $\displaystyle{y_k =\frac{\alpha_k\gamma_kv_k+\gamma_{k+1}x_k}{\gamma_k+\mu\alpha_k}}$.
		\STATE Update $x_{k+1} $ such that $\displaystyle{f(x_{k+1})\leqslant f(y_{k})-\frac{1}{2L}\nm{\nabla f(y_{k})}_{*}^2}$.
		\label{algo:NAG-update-xk}
		\STATE Update $\displaystyle{v_{k+1} =
			\frac{1}{\gamma_{k+1}} \left [(1-\alpha_k)\gamma_kv_k+\alpha_k(\mu y_k- \nabla f(y_k))\right ]}$.	
		\ENDFOR
	\end{algorithmic}
\end{algorithm}
Note that we have many choices for $x_{k+1}$ in step \ref{algo:NAG-update-xk} of Algorithm \ref{algo:NAG}. One noticeable example is the gradient descent step (see~\cite[Chapter 2, Constant Step Scheme, I]{Nesterov:2013Introductory}):
\begin{equation}\label{eq:update-xk}
	x_{k+1} = y_k - \frac{1}{L}\nabla f(y_k).
\end{equation}
With this choice,
the sequence $\{v_k\}$ in Algorithm 
\ref{algo:NAG} can be eliminated and $y_{k+1}$ is updated by that 
(see~\cite[Chapter 2, Constant Step Scheme, II]{Nesterov:2013Introductory})
\[
y_{k+1} = x_{k+1} + \frac{\alpha_k-\alpha_k^2}{\alpha_{k+1}+\alpha_k^2}
(x_{k+1}-x_k),
\]
where $\alpha_{k+1}\in(0,1)$ is calculated from the quadratic equation
\[
L\alpha_{k+1}^2 = L\alpha^2_k(1-\alpha_{k+1})+\mu\alpha_{k+1}.
\]
If $\mu>0$ and $\alpha_0=\sqrt{\mu/L}$, then $\alpha_k = \sqrt{\mu/L}$; see~\cite[Chapter 2, Constant Step Scheme, III]{Nesterov:2013Introductory}.
In particular, if $\mu=0$, then Algorithm \ref{algo:NAG} (with $x_{k+1}$ updated by~\eqref{eq:update-xk}) coincides with the accelerated scheme proposed by Nesterov early in the 1980s~\cite{Nesterov1983}.

\subsection{NAG method as a corrected  semi-implicit scheme}
After simple calculations, we can rewrite 
Algorithm \ref{algo:NAG} as an equivalent form
\begin{equation}\label{eq:ex2-ode-NAG-1}
	\left\{
	\begin{aligned}
		\frac{\gamma_{k+1}  - \gamma_{k}  }{\alpha_k}
		={}&\mu-\gamma_{k},\\
		\frac{y_{k}-x_{k}}{\alpha_k}={}& 
		\frac{\gamma_k}{\gamma_{k+1}}(v_{k}-y_k),\\
		\frac{v_{k+1}-v_{k}}{\alpha_k}={}&
		\frac{\mu}{\gamma_{k+1}}(y_k-v_{k})
		-\frac{1}{\gamma_{k+1}}\nabla f(y_k),
	\end{aligned}
	\right.
\end{equation}
where in addition we update $x_{k+1}$ satisfying
\begin{equation}\label{eq:ex2-ode-NAG-2}
	f(x_{k+1})\leqslant f(y_k)-
	\frac{1}{2L}\nm{\nabla f(y_k)}_{*}^2.
\end{equation}
Surprisingly,~\eqref{eq:ex2-ode-NAG-1} formulates a semi-implicit discretization 
for our NAG flow~\eqref{eq:ode-sys-NAG} with a correction step \eqref{eq:ex2-ode-NAG-2} and an explicit discretization for the equation~\eqref{eq:gammat-NAG} of $\gamma$. Similar to~\eqref{eq:ex-1-NAG}, we can adopt the gradient descent step 
which promises~\eqref{eq:ex2-ode-NAG-2}.

Based on subtle algebraic calculations of the estimate sequence, Nesterov~\cite[Chapter 2]{Nesterov:2013Introductory} proved the convergence rate of Algorithm \ref{algo:NAG}. 
In the following, we give an alternative proof by using the Lyapunov function~\eqref{eq:Lk}.
\begin{theorem}
	\label{thm:conv-ex2-ode-NAG}
	Assume that 	$f\in\mathcal {S}_{\mu,L}^{1,1}$ 
	with $0\leqslant \mu\leqslant L<\infty$.
	If $L\alpha_k^2=  \gamma_{k+1}$, 
	then for Algorithm \ref{algo:NAG}, i.e., the scheme 
	\eqref{eq:ex2-ode-NAG-1} together with~\eqref{eq:ex2-ode-NAG-2}, we have $0<\alpha_k\leqslant 1$ and
	\begin{equation}
		\label{eq:decay-Lk-ex2}
		\mathcal L_{k+1}\leqslant (1-\alpha_k)\mathcal L_k,
		\quad k\in\mathbb N,
	\end{equation}
	where $\mathcal L_k$ is defined by~\eqref{eq:Lk}.
	Consequently for all $k\geqslant 0$,
	\begin{equation}\label{eq:rate-Lk-ex2}
		\mathcal L_{k}\leqslant {}	\mathcal L_{0}\times
		\min\left\{
		\frac{4L}{(\sqrt{\gamma_0}\, k+2\sqrt{L})^2},\,
		\left(1-\sqrt{\frac{\min\{\gamma_1,\mu\}}L}\right)^{k}
		\right\}.
	\end{equation}
	Moreover, for all $k\geqslant 1$,
	\begin{equation}
		\label{eq:rate-Lk-ex0-L-NAG}
		\begin{split}
			\mathcal L_{k}\leqslant {}&
			C_{\gamma_0,L}\times
			\min\left\{
			\frac{4}{k^2},\,
			\left(1-\sqrt{\frac{\min\{\gamma_1,\mu\}}L}\right)^{k-1}
			\right\},
		\end{split}
	\end{equation}
	where $C_{\gamma_0,L}$ has been defined in \eqref{eq:CgL}.
\end{theorem}
\begin{proof}
	Let us first prove~\eqref{eq:decay-Lk-ex2}. 
	By~\eqref{eq:ex2-ode-NAG-1}, we find
	\[
	\left\{
	\begin{split}
		{}& 	v_{k} =y_{k}+
		\frac{\gamma_{k+1}}
		{\alpha_k\gamma_k}(y_{k}-x_{k}),\\
		{}&	v_{k+1} =y_k+\frac{1-\alpha_k}{\alpha_k}(y_k-x_k)
		-\frac{\alpha_k}{\gamma_{k+1}}\nabla f(y_k),
	\end{split}
	\right.
	\]
	and a direct computation gives
	\begin{equation*}
		\begin{split}
			{}&\frac{\gamma_{k+1}}{2}
			\nm{v_{k+1}-x^*}^2 - 
			\frac{\gamma_{k}}{2}(1-\alpha_k)
			\nm{v_{k}-x^*}^2 \\
			={}&
			\alpha_k
			\left(
			\dual{\nabla f(y_k),x^*-y_k}+
			\frac{\mu }{2}\nm{x^*-y_k}^2
			\right)
			\\
			{}&\quad +
			(1-\alpha_k)
			\left(
			\dual{\nabla f(y_k),x_k-y_k}+
			\frac{\mu }{2}\nm{x_k-y_k}^2
			\right)\\
			{}&\qquad+ \frac{\alpha_k^2}{2\gamma_{k+1}}\nm{\nabla f(y_k)}_*^2-\frac{\mu(1-\alpha_k)}{2\alpha_k\gamma_k}
			(\gamma_{k}+\mu\alpha_k)
			\nm{y_k-x_k}^2.
		\end{split}
	\end{equation*}
	Dropping the negative term $-\nm{y_k-x_k}^2$ and using the $\mu$-convexity of $f$ imply that
	\begin{equation*}
		\begin{split}
			{}&\frac{\gamma_{k+1}}{2}
			\nm{v_{k+1}-x^*}^2 - 
			\frac{\gamma_{k}}{2}(1-\alpha_k)
			\nm{v_{k}-x^*}^2 \\
			\leqslant {}&		\alpha_k
			\left(
			f(x^*)-f(y_k)
			\right)+		(1-\alpha_k)
			\left(f(x_k)-f(y_k)\right)
			+
			\frac{\alpha_k^2}{2\gamma_{k+1}}\nm{\nabla f(y_k)}_*^2,
		\end{split}
	\end{equation*}
	and we get the inequality
	\begin{equation*}
		\begin{split}
			\mathcal L_{k+1}-(1-\alpha_k)\mathcal L_{k} 
			\leqslant{}&
			f(x_{k+1})-f(y_{k})+ 
			\frac{\alpha_k^2}{2\gamma_{k+1}}\nm{\nabla f(y_k)}_{*}^2.
		\end{split}
	\end{equation*}
	Consequently, by~\eqref{eq:ex2-ode-NAG-2} and the relation $L\alpha_k^2=\gamma_{k+1}$, 
	the right hand side of the above inequality is negative, which proves ~\eqref{eq:decay-Lk-ex2}. 
	
	In this case, we modify \eqref{eq:rhok} as follows
	\begin{equation}\label{eq:rhok-1-a}
		\rho_0=1,\quad\rho_k:=		\prod_{i=0}^{k-1}(1-\alpha_i),\quad k\geqslant 1,
	\end{equation}
	then by~\eqref{eq:decay-Lk-ex2} it is clear that $		\mathcal L_k\leqslant \rho_k\mathcal L_0$,
	and invoking Lemma \ref{lem:gk-ak} 
	proves~\eqref{eq:rate-Lk-ex2}. As the proof of 
	\eqref{eq:rate-Lk-ex0-L-NAG} is very similar with that 
	of~\eqref{eq:rate-Lk-ex0-L}, we omit the details 
	here and conclude the proof of this theorem.		
	 \end{proof}
\begin{remark}\label{rem:compare}
	Similar to our corrected
	schemes~\eqref{eq:NAG-GS-1st-extra} and~\eqref{eq:ex-1-NAG}, 
	NAG method (i.e., Algorithm \ref{algo:NAG}) generates a three-term
	sequence $\{(x_k,y_{k},v_{k})\}$ as well.
	If $\mu=0$, then they share the same convergence rate bound
	\[
	\mathcal L_k\leqslant \frac{4L	 \mathcal L_0}{(\sqrt{\gamma_0}\, k+2\sqrt{L})^2},
	\]
	and when $\gamma_0=\mu>0$, we have
	\begin{equation}\label{eq:Lk-compa}
		\mathcal L_k\leqslant  \mathcal L_0\times
		\left\{
		\begin{aligned}
			&(1-\sqrt{\mu/L})^{k},&&\text{for NAG method},\\
			&(1+\sqrt{\mu/L})^{-k},&&\text{for~\eqref{eq:ex-1-NAG} and~\eqref{eq:NAG-GS-1st-extra} }.
		\end{aligned}
		\right.
	\end{equation}
	In view of the trivial fact 
	$$
	1 - \epsilon = \frac{1}{1+\epsilon} - \frac{\epsilon^2}{1+\epsilon}, \quad \epsilon =  \sqrt{\mu/L}\leqslant 1,
	$$
	we see the rates in~\eqref{eq:Lk-compa} are asymptotically the same 
	but NAG method can achieve a slightly better convergence rate. 
	However, we note that they share the same computational complexity
	\[
	O\left(
	\min\big\{
	\sqrt{L/\epsilon},\,\sqrt{L/\mu}\cdot|\ln\epsilon|
	\big\}
	\right),
	\]
	which is optimal, in the sense that~\cite{Nesterov:2013Introductory} it achieves the complexity lower bound of first-order algorithms for the function class $\mathcal S_{\mu,L}^{1,1}$ with $0\leqslant \mu\leqslant L<\infty$.
	
\end{remark}
\begin{remark}
	Unlike the gradient descent method, 
	the function value $f(x_k)$ of accelerated gradient methods may not decrease in each step. It is the discrete Lyapunov function $\mathcal L_k$ that is always decreasing; see \eqref{eq:decay-Lk-ex0}, \eqref{eq:decay-Lk-ex1} and \eqref{eq:decay-Lk-ex2}.
	
\end{remark}
\begin{remark}
	To reduce the function value, one can adopt the restating strategy~\cite{odonoghue_adaptive_2015}. Specifically, given $(\gamma_0,v_0,x_0)$, if $f(x_k)$ is increasing after $k$-iteration, then set $k=0$ and restart the iteration process with another initial guess $(\tilde{\gamma}_0,\tilde{v}_0,\tilde{x}_0)$.
	By Theorems \ref{thm:conv-ex0-ode-NAG}, \ref{thm:conv-ex1-ode-NAG} and \ref{thm:conv-ex2-ode-NAG}, 
	when $f\in\mathcal {S}_{0,L}^{1,1}$ and 
	$\gamma_0=L,v_0 = x_0$, we only have the sublinear convergence rate 
	\begin{equation}\label{eq:obj-rate}
		f(x_k)-f(x^*)
		\leqslant \frac{4}{k^2}\!
		\left(
		f(x_0)-f(x^*)+\frac{L}{2}
		\nm{x_0-x^*}^2
		\right)\!\!
		\leqslant \frac{4L}{k^2}
		\nm{x_0-x^*}^2,
	\end{equation}
	where we used~\eqref{eq:Lip-L-equiv}, which promises
	\[
	f(x_0)-f(x^*)\leqslant \frac{L}{2}\nm{x_0-x^*}^2.
	\]
	Additionally, assume $f$ satisfies the {\it quadratic growth condition} with $\sigma>0$:
	\[
	f(x)-f(x^*)\geqslant \sigma{\rm dist}^2(x,\argmin f)
	\quad\forall\,x\in V,
	\] 
	where ${\rm dist}(x,\argmin f) = \inf_{x^*\in\argmin f}\nm{x-x^*}$. 
	As \eqref{eq:obj-rate} holds for all $x^*\in\argmin f$, 
	we have immediately that
	\[
	f(x_k)-f(x^*)	\leqslant{} \frac{4L}{k^2}{\rm dist}^2(x,\argmin f)
	\leqslant 
	\frac{4L }{\sigma k^2}
	(f(x_0)-f(x^*)).
	\]
	Therefore, as analyzed in~\cite{necoara_linear_2019}, 
	if we consider fixed restart technique~\cite{odonoghue_adaptive_2015} 
	every $k$ steps, then after $N=nk$ steps we will get
	\[
	f(x_{N})-f(x^*)\leqslant \left(\frac{4L}
	{\sigma k^2}\right)^{n}(f(x_0)-f(x^*)).
	\]
	Evidently, the optimal choice $	k_{\#} =e \sqrt{4L/\sigma}$	
	yields the linear rate
	\[
	f(x_{N})-f(x^*)\leqslant 
	e^{-2N/k_{\#}}
	(f(x_0)-f(x^*)).
	\]	
	If the parameter $\sigma$ is unknown, one can use the adaptive 
	restart technique~\cite{odonoghue_adaptive_2015}. 
	
	When $f$ is quadratic and convex, changing $\gamma_k$ from $L$ to $\mu$ periodically will smoothing out error in different frequencies and can further optimize the constant in front of the accelerated rate. That is, the dynamically changing parameter $\{\gamma_k\}$ hopefully outperforms the fixed one $\gamma_k=\mu$. For general nonlinear convex functions, a rigorous justification of the restart strategy is under investigation. 
\end{remark}	
\section{Composite Convex Optimization}
\label{sec:comp}
In this part we mainly focus on the composite optimization
\begin{equation}\label{min-comp}
	\min_{x\in Q} f(x):=
	\min_{x\in Q} \left [h(x)+g(x) \right ],
\end{equation}
where $Q\subseteq V$ is a simple closed convex set, $h\in\mathcal S_{\mu,L}^{1,1}(Q)$ with $0\leqslant \mu\leqslant L<\infty$ 
and $g:V\to\mathbb R\cup\{+\infty\}$ is proper, closed and 
convex, and $Q\cap {\bf dom}\, g\neq\emptyset$. In general $g$ is not differentiable but its subdifferential $\partial g$ exists as a set-valued function. 
More precisely, the subdifferential $\partial g(x)$ of $g$ at $x$ is defined by that
\begin{equation}\label{eq:def-subg}
	\partial g(x) :=\left\{p\in V^*:\,g(y)\geqslant g(x)+\dual{p,y-x}\quad
	\,\forall\, y\in V\right\}.
\end{equation}

\begin{remark}\label{rem:f-g-fb}
	For the case that $h\in\mathcal S_{0,L}^{1,1}(Q)$ and $g$ is $\mu$-strongly convex 
	with $\mu\geqslant 0$, we can split $h+g$ as $(h(x) + \frac{\mu}{2}\|x\|^2) + (g(x) - \frac{\mu}{2}\|x\|^2)$, which reduces to our current assumption for~\eqref{min-comp}.
	
\end{remark}

We shall apply our ODE solver approach to the problem~\eqref{min-comp}. 
The first step is to generalize the dynamical system~\eqref{eq:ode-sys-NAG} to the current nonsmooth 
setting. Basically, we set $F = f+i_Q$ with $i_Q$ being the indicator function of $Q$ and obtain 
a differential inclusion for minimizing $F$ on $V$, which is equivalent to 
minimize $f$ over $Q$. After that, optimization methods (see Algorithms \ref{algo:APG} and \ref{algo:AFB}) for solving the original problem~\eqref{min-comp} with the accelerated convergence rate 
\[
O\left(\min\big\{L/k^2,(1+\sqrt{\mu/L})^{-k}\big\}\right)
\]
are proposed from numerical discretizations of the continuous model~\eqref{eq:ADI-x-v}. This is a proof of the effective and usefulness of our NAG flow model~\eqref{eq:ADI-x-v} and the ODE solver approach, by which we can construct new accelerated methods. 

\subsection{Continuous model}
For minimizing a nonsmooth function $F$ over $V$, our NAG 
flow~\eqref{eq:ode-sys-NAG} becomes a differential inclusion
\begin{equation}\label{eq:ADI-x-v}
	\left\{
	\begin{aligned}
		x'={}& v - x,\\
		\gamma v'\in{}&
		\mu (x - v )-\partial F(x).
	\end{aligned}
	\right.
\end{equation}
To ensure solution existence, suitable initial conditions shall be imposed later. Correspondingly, the second-order ODE~\eqref{eq:ode-NAG} reads as a 
second-order differential inclusion
\begin{equation}\label{eq:ADI-x}
	\gamma x''+(\mu+\gamma)x'+\partial F(x)\ni 0.
\end{equation}
Above, the scaling factor $\gamma$ is still the solution to~\eqref{eq:gammat-NAG}.

As the subdifferential $\partial F$ is a set-valued maximal monotone 
operator, classical $C^2$ solution to~\eqref{eq:ADI-x} may not exist 
because discontinuity can occur in $x'$. Therefore, the concept 
of {\it energy-conserving solution} has been introduced in~\cite{cabot_asymptotics_2007,paoli_existence_2000,schatzman_class_1978}. 

Let us assume the initial data
\begin{equation}\label{eq:x0-x1}
	x (0) = x_0\in{\bf dom} F\quad\text{and}\quad x'(0) 
	= x_1\in \mathcal T_{{\bf dom} F}(x_0),
\end{equation}
where $\mathcal T_{{\bf dom} F}(x_0)$ denotes the tangent 
cone of ${\bf dom} F$ at $x_0$:
\[
\mathcal T_{{\bf dom} F}(x_0):=
\overline{\mathop{\cup}\limits_{\tau>0}\tau(x_0-\overline{{\bf dom} F})}.
\]
In addition, we shall introduce some vector-valued functional spaces.
Given any interval $I\subset \mathbb R$, let $M(I;V)$ be the space of $V$-valued Radon measures on $I$; for any $m\in\mathbb N$ and $1\leqslant p\leqslant\infty$, $W^{m,p}(I;V)$ denotes the standard $V$-valued Sobolev space~\cite{Kreuter2015}; the space of all $V$-valued functions with bounded variation is defined by $BV(I;V)$~\cite{attouch_variational_2014}. 
Also, $W_{loc}^{m,p}(I;V)$ and $BV_{loc}(I;V)$ consist of all the sets $W^{m,p}(\omega;V)$ and $BV(\omega;V)$ respectively, where $\omega\subset I$ is any compact subset.
\begin{Def}\label{def:solu}
	We call $x:[0,\infty)\to V$ an energy-conserving
	solution to~\eqref{eq:ADI-x} 
	with initial data~\eqref{eq:x0-x1} if it satisfies the following.
	\begin{enumerate}
		\item $x\in W^{1,\infty}_{loc}(0,\infty;V),\,x(0) = x_0$ 
		and $x(t)\in {\bf dom} F$ for all $t>0$.
		\item $x'\in BV_{loc}([0,\infty);V),\,x'(0+) = x_1$.		
		\item For almost all $t>0$, there holds the energy equality:
		\[
		\begin{aligned}
			{}&	F(x(t)) + \frac{\gamma(t)}2\nm{x'(t)}^2
			+
			\int_{0}^{t}\frac{\mu+3\gamma(s)}{2}\nm{x'(s)}^2
			{\mathrm d}s
			= {} F(x_0) + \frac{\gamma_0}2\nm{x_1}^2.
		\end{aligned}
		\]
		\item There exists some $\nu\in M(0,\infty;V)$ such that 
		\[
		\gamma x'' + (\mu+\gamma)x'+\nu = 0
		\]
		holds in the sense of distributions, and for any $T>0$, we have
		\[
		\int_{0}^{T}\big(F(y(t))-F(x(t))\big){\mathrm d}t
		\geqslant\dual{	\nu,y-x}_{C([0,T];V)}
		\quad \text{for all } y\in C([0,T];V).
		\]
	\end{enumerate}
\end{Def}

In~\cite{luo_accelerated_2021}, the problem~\eqref{eq:ADI-x} has been extended to a general case 
\[
\gamma x''+(\mu+\gamma)x'+\partial F(x)\ni \xi,
\]
where $\xi$ stands for small perturbation.
Therefore, according to~\cite[Theorem 2.1]{luo_accelerated_2021}, we have the 
existence of an energy-conserving solution to~\eqref{eq:ADI-x} and by~\cite[Theorems 2.2 
and 2.3]{luo_accelerated_2021}, we obtain the exponential decay, which is a nonsmooth version of \eqref{eq:conv-ode-NAG}.
\begin{theorem}\label{thm:regu-ode-OAG}
	Assume V is a finite dimensional Hilbert space. In the sense of Definition \ref{def:solu}, the differential inclusion~\eqref{eq:ADI-x} admits an energy-conserving 
	solution $x:[0,\infty)\to V$ satisfying
	\begin{equation}\label{eq:ex-comp}
		F(x(t))-F(x^*)
		+\frac{\gamma(t)}{2}\nm{x(t)+x'(t)-x^*}^2
		\leqslant 2	\mathcal L_0e^{-t},
	\end{equation}		
	for almost all $t>0$, where $	\mathcal L_0 := F(x_0)-F(x^*)
	+\frac{\gamma_0}{2}\nm{x_0+x_1-x^*}^2$.
\end{theorem}
\begin{remark}
	If 
	additionally ${\bf dom} F = V$, then $x\in W^{2,\infty}_{loc}(0,\infty;V)
	\cap C^1([0,\infty);V)$ and~\eqref{eq:ex-comp} holds for all $t>0$. 
\end{remark}
\subsection{An APGM for unconstrained optimization}
\label{sec:APG}
Let us first consider the unconstrained case $Q=V$, i.e.,
\begin{equation}\label{min-comp-V}
	\min_{x\in V} f(x):=
	\min_{x\in V} \left [h(x)+g(x) \right ],
\end{equation}
where $f\in\mathcal S_{\mu,L}^{1,1}$ with $0\leqslant \mu\leqslant L<\infty$ and $g:V\to\mathbb R\cup\{+\infty\}$ is a properly closed and convex function and possibly nonsmooth.
\subsubsection{Gradient mapping}
\label{sec:APG-gdmap}	
To treat the nonsmooth part $g$, we introduce the tool of gradient mapping. 
Following~\cite[Chapter 2]{Nesterov:2013Introductory}, given any $\eta >0$, 
the composite gradient mapping
$\mathcal G_f(x,\eta)$ of $f$ at $x$
is defined by that
\begin{equation}\label{eq:gd-map}
	\mathcal G_f(x,\eta):=
	\frac{x-S_f(x,\eta)}{\eta}
	\quad x\in V,
\end{equation}
where $S_f(x,\eta):=\prox_{\eta g}(x-\eta \nabla h(x))$ and the proximal operator $\prox_{\eta g}$ has been defined by~\eqref{eq:prox}. 
Note that $S_f(x,\eta)$ is clearly well-defined and so is $\mathcal G_f(x,\eta)$. 

It is well known~\cite{Parikh:2014,rockafellar_convex_1970} that 
\begin{equation}\label{eq:prox-subg}
	\frac{x-\prox_{\eta g}(x)}{\eta}\in \partial g(	\prox_{\eta g}(x)) ,
\end{equation}
which yields the fact 
\begin{equation}\label{eq:sub-gradient}
	\mathcal G_f(x,\eta)-\nabla h(x)\in
	\partial g(S_f(x,\eta)).
\end{equation}
From this we conclude that the fixed-point set of $S_f(\cdot,\eta)$ is $\argmin f$. Indeed, $x = S_f(x,\eta)$ if and only if $0\in\partial f(x) $. We also observe from \eqref{eq:sub-gradient} that the gradient mapping~\eqref{eq:gd-map} is defined reversely from the proximal-gradient step for minimizing $f = h+g$, i.e., 
\[
\frac{S_f(x,\eta)-x}{\eta}\in- \nabla h(x) - \partial g(S_f(x,\eta)) = -\mathcal G_f(x,\eta).
\]
Hence it plays the 
role of the gradient $\nabla f$ in the smooth case.
Particularly, if $g = 0$, then 
$\mathcal G_f(x,\eta) = \nabla h(x)$ and $S_f(x,\eta)=x-\eta \nabla h(x)$ is 
nothing but a gradient step.

To move on, we present an auxiliary lemma, 
which is a key ingredient for our convergence analysis. 
As we will fix $\eta=1/L$, for simplicity, we set $\mathcal G_f(x):=
\mathcal G_f(x,1/L)$ and $S_f(x):=S_f(x,1/L)$. 
\begin{lemma}\label{lem:key-gd-map}
	Assume $f = h+g$, where $h\in\mathcal S_{\mu,L}^{1,1}$ with $0\leqslant \mu
	\leqslant L<\infty$ and $g:V\to\mathbb R\cup\{+\infty\}$ is properly closed 
	and convex. Then for any $x,y\in V$, 
	\begin{equation}\label{eq:key-gd-map}
		\begin{split}
			f(y)\geqslant{}&
			f(S_f(x))
			+\dual{\mathcal G_f(x),y-x}
			+\frac{\mu}{2}\nm{y-x}^2
			+\frac{1}{2L}
			\nm{\mathcal G_f(x)}^2.
		\end{split}
	\end{equation}
\end{lemma}
\begin{proof}
	Since $h\in\mathcal S_{\mu,L}^{1,1}$, applying~\eqref{eq:convex-mu} 
	and~\eqref{eq:Lip-L-equiv} gives
	\[
	\begin{split}
		h(x)-h(y)+\dual{\nabla h(x),y-x}
		\leqslant{}& -\frac{\mu}{2}\nm{x-y}^2,\\
		h(S_f(x))-	h(x)+\dual{\nabla h(x),x-S_f(x)}
		\leqslant{}& \frac{L}{2}\nm{S_f(x)-x}^2,
	\end{split}
	\]
	which implies that
	\[
	\begin{split}
		h(y)\geqslant{}&
		h(S_f(x,\eta))
		+\dual{\nabla h(x),y-S_f(x)}
		+\frac{\mu}{2}\nm{y-x}^2-\frac{1}{2L}
		\nm{\mathcal G_f(x)}^2.
	\end{split}
	\]
	Observing~\eqref{eq:sub-gradient}, we get
	\[
	\begin{split}
		g(y)\geqslant {}&
		g(S_f(x))
		+\dual{ \mathcal G_f(x)-\nabla h(x),
			y-S_f(x)}.
	\end{split}
	\]
	Summing the above two inequalities and using the split
	\begin{align*}
		\dual{ \mathcal G_f(x),
			y-S_f(x)}  &= 	\dual{ \mathcal G_f(x),
			y-x} + 	\dual{ \mathcal G_f(x),
			x-S_f(x)}\\
		&= 	\dual{ \mathcal G_f(x),
			y-x} + \frac{1}{L}
		\nm{\mathcal G_f(x)}^2,
	\end{align*}
	we finally arrive at~\eqref{eq:key-gd-map} and end the proof of this lemma.
	 \end{proof}
\begin{remark}
	For a fixed $x$, the right hand side of~\eqref{eq:key-gd-map} defines a quadratic approximation of $f$ at $x$, and it is strongly reminiscent of the quadratic lower bound approximation \eqref{eq:convex-mu} for the smooth case. However, compared to \eqref{eq:convex-mu}, the constant is shifted from $f(x)$ to a lower value $f(S_f(x)) + \frac{1}{2L} \nm{\mathcal G_f(x)}^2$. The first order part is $\mathcal G_f(x)$ instead of the subgradient at $x$. The quadratic part $\frac{\mu}{2}\nm{y-x}^2$ is due to the $\mu$-convexity.	
\end{remark}
\subsubsection{The proposed  method}
Based on the corrected semi-implicit scheme~\eqref{eq:ex-1-NAG} for 
NAG flow~\eqref{eq:ode-sys-NAG}, it is possible to 
generalize it to solve the differential inclusion~\eqref{eq:ADI-x-v}. 
Indeed, we just replace the gradient $\nabla f(y_k)$ with the gradient mapping $\mathcal G_f(y_k)$ and set the correction as $	x_{k+1} ={}S_f(y_k)$. More precisely, consider
\begin{equation}\label{eq:ex1-ode-OAG}
	\left\{
	\begin{aligned}
		\frac{y_k-x_{k}}{\alpha_k}={}& v_{k}-y_k,\\
		x_{k+1} ={}&S_f(y_k),\\
		\frac{v_{k+1}-v_{k}}{\alpha_k}={}&
		\frac{\mu}{\gamma_k}(y_k-v_{k+1})
		-\frac{1}{\gamma_k}
		\mathcal G_f(y_{k}),\\
		\frac{		\gamma_{k+1}  - \gamma_{k}}{\alpha_k}  
		={}&\mu-\gamma_{k+1}.
	\end{aligned}
	\right.
\end{equation}

Once $x_{k+1}=S_f(y_k)=\prox_{\eta g}(y_k-\eta\nabla h(y_k))$ is obtained, we can update $v_{k+1}$ with known datum $x_k,y_k,v_k$ and $x_{k+1}$. Thus in each iteration,~\eqref{eq:ex1-ode-OAG} only calls the proximal operation $\prox_{\eta g}$ once. 

We still use the step size $L\alpha_k^2=\gamma_k(1+\alpha_k)$ and summarize the semi-implicit scheme~\eqref{eq:ex1-ode-OAG} in Algorithm \ref{algo:APG}, which is called semi-implicit APGM (Semi-APGM for short). Also, the convergence rate is derived via the discrete Lyapunov function \eqref{eq:Lk}.
\begin{algorithm}[H]
	\caption{Semi-APGM for solving $\min_{x\in V} \left [h(x)+g(x) \right ]$}
	\label{algo:APG}
	\begin{algorithmic}[1] 
		\REQUIRE  $x_0,v_0\in V,\,\gamma_0>0$ and $\eta = 1/L$.
		\FOR{$k=0,1,\ldots$}
		\STATE Compute $\alpha_k>0$ such that $L\alpha_k^2= \gamma_k\big(1+\alpha_k\big)$.
		\STATE Update $\displaystyle{\gamma_{k+1} = \frac{\gamma_k+\mu\alpha_k}{1+\alpha_k}}$.
		\STATE Set $\displaystyle{y_k =\frac{x_k+\alpha_kv_k}{1+\alpha_k}}$ and $\displaystyle{w_k =\frac{\gamma_kv_k+\mu\alpha_ky_k}{\gamma_k+\mu\alpha_k}}$.
		\STATE Update $\displaystyle{x_{k+1} =\prox_{\eta g}(y_k-\eta\nabla h(y_k))}$.
		\STATE Set $\displaystyle{v_{k+1} = w_k+\frac{\gamma_k}{\gamma_{k+1}}
			\frac{x_{k+1}-y_k}{\alpha_k}}$.			
		\ENDFOR
	\end{algorithmic}
\end{algorithm}
\begin{theorem}\label{thm:conv-APG}
	For Algorithm \ref{algo:APG}, we have
	\begin{equation}\label{eq:conv-APG}
		\mathcal L_{k+1}
		\leqslant 
		\frac{	 \mathcal L_k}{1+\alpha_k }\quad\forall\,k\in\mathbb N,
	\end{equation}
	where $\mathcal L_k = 		f(x_k)-f(x^*) +
	\frac{\gamma_k}{2}
	\nm{v_k-x^*}^2
	$, and both~\eqref{eq:rate-Lk-ex0} and 
	\eqref{eq:rate-Lk-ex0-L} hold true here.
\end{theorem}
\begin{proof}
	The proof of~\eqref{eq:conv-APG} 
	is very similar to that of~\eqref{eq:decay-Lk-ex1}. 
	Indeed, replacing $x_{k+1}$ and its gradient $\nabla f(x_{k+1})$ in \eqref{eq:ex1-ode-NAG} respectively with $y_k$ and $\mathcal G_f(y_k)$, we can proceed as the proof of Lemma \ref{lem:bd-NAG-GS-1st} and use Lemma \ref{lem:key-gd-map} to obtain
	\begin{equation}\label{eq:conv-mid-APG}
		\begin{aligned}
			\widehat{		\mathcal L}_{k}-\mathcal L_{k} 
			\leqslant{}&-\alpha_k \widehat{		\mathcal L}_{k}
			+(1+\alpha_k)\left(f(y_k)-f(x_{k+1})\right)\\
			{}&\quad	+	\frac{\alpha_k^2}{2\gamma_k}
			\nm{\mathcal G_f(y_k)}^2-\frac{1+\alpha_k}{2L}\nm{\mathcal G_f(y_k)}^2,
		\end{aligned}
	\end{equation}
	where $\widehat{		\mathcal L}_{k}$ is defined by~\eqref{eq:hat-Lk}.
	Thanks to the relation $L\alpha_k^2=\gamma_k(1+\alpha_k)$, the second line of \eqref{eq:conv-mid-APG} vanishes, and inserting the identity $f(y_k)-f(x_{k+1})=	\widehat{		\mathcal L}_{k}-		\mathcal L_{k+1}$ into \eqref{eq:conv-mid-APG} 
	gives \eqref{eq:conv-APG}. Based on this, it is not hard to see that both~\eqref{eq:rate-Lk-ex0} and~\eqref{eq:rate-Lk-ex0-L} hold true. 
	This finishes the proof of this theorem.
	 \end{proof}

We mention that with another choice 
\[
L\alpha_k^2 = \mu\alpha_k^2+\gamma_{k}(1+\alpha_k),
\]
we can drop the sequence $\{v_{k}\}$ from~\eqref{eq:ex1-ode-OAG}. 
The procedure is not straightforward but very similar 
to that of Nesterov's optimal method in~\cite[page 80]{Nesterov:2013Introductory}. 
We omit the details and only list the following algorithm.
\begin{algorithm}[H]
	\caption{Simplified Semi-APGM}
	\label{algo:simple-APG}
	\begin{algorithmic}[1] 
		\REQUIRE  $x_0,y_0\in V,\,\gamma_0>0$ and $\eta = 1/L$.
		\FOR{$k=0,1,\ldots$}
		\STATE Compute $\alpha_k>0$ such that $L\alpha_k^2 = \mu\alpha_k^2+\gamma_{k}(1+\alpha_k)$.
		\STATE Update $\displaystyle{\gamma_{k+1} = \frac{\gamma_k+\mu\alpha_k}{1+\alpha_k}}$ and set $\beta_k
		=\frac{L\alpha_k}{\gamma_{k+1}(1+\alpha_k)}$.
		\STATE Set $\displaystyle{y_{k+1} ={}x_{k}+\beta_k(x_{k+1}-x_k)}$.
		\STATE Update $\displaystyle{x_{k+1} =\prox_{\eta g}(y_k-\eta\nabla h(y_k))}$.
		\ENDFOR
	\end{algorithmic}
\end{algorithm}
This can be viewed as a generalization of~\cite[Chapter 2, Constant Step Scheme, II]{Nesterov:2013Introductory} to problem~\eqref{min-comp-V}. Particularly, for convex case $\mu=0$, it is very close to FISTA~\cite{Beck2009}.
Both of them share the same spirit: applying one proximal gradient step first and then using some extrapolation formulae.
The difference comes only from the use of the two sequences $\{\alpha_k\}$ and $\{\beta_k\}$. 
We also claim that Algorithm \ref{algo:simple-APG} has the same accelerated convergence rate as Algorithm \ref{algo:APG}, i.e., $O(\min(L/k^2,(1+\sqrt{\mu/L})^{-k}))$. In contrast  FISTA is designed for $\mu =0$ and has only the sublinear rate $O(L/k^2)$.

We also mention that, accelerated proximal gradient methods for solving~\eqref{min-comp-V} with only one 
evaluation of ${\bf prox}_{\eta g}$ 
in each iteration can be found in~\cite{Siegel:2019} (only for strongly convex case) and~\cite[Chapter 2, Algorithm 2.2]{Lin2020} (for both convex and strongly convex cases).

Both Algorithms \ref{algo:APG} and~\ref{algo:simple-APG} cannot be applied 
directly to the general constraint case~\eqref{min-comp}. The main issue comes from
the definition~\eqref{eq:gd-map} of the gradient mapping $\mathcal G_f(x,\eta)$, 
where we shall impose the restriction $x\in Q$ and calculate the proximal operator $\prox_{\eta g}$ over $Q$ to obtain $S_f(x)\in Q$. For both two algorithms, we shall compute $x_{k+1}=S_f(y_k)=\prox_{\eta g}(y_k-\eta\nabla h(y_k))$. But the sequence $\{y_k\}$
in Algorithms \ref{algo:APG} and~\ref{algo:simple-APG} may be outside the constraint set. This is not acceptable 
because $\nabla h(y_k)$ might not exist: for instance, $Q = [0,\infty)$ and $h$ is the entropy function.

The original FISTA~\cite{Beck2009} and the methods in~\cite{Siegel:2019}  and~\cite[Chapter 2, Algorithm 2.2]{Lin2020} mentioned above, cannot be applied to the constrained problem~\eqref{min-comp} either. 
This stimulates us to propose a new operator splitting scheme to conquer this problem.

\subsection{An accelerated forward-backward method for constrained optimization}
\label{sec:AFB}
We now go back to the constrained problem~\eqref{min-comp}. As mentioned above, the tool of gradient mapping is not convenient for us to handle this case. To avoid using it, we utilize the separable structure of $f=h+g$ and apply explicit and implicit schemes for $h$ and $g$, respectively. This is the so-called 
{\it operator splitting} technique in ODE solvers and is also known as the forward-backward method. 

Let us start from the predictor-corrector scheme 
\eqref{eq:NAG-GS-1st-extra} and rewrite it as follows
\begin{equation}
	\left\{
	\begin{aligned}
		{}&y_k = \frac{x_k+\alpha_kv_k}{1+\alpha_k},\quad
		w_k = {}\frac{\gamma_kv_k+\mu\alpha_ky_k}{\gamma_k+\mu\alpha_k},\\
		{}&v_{k+1} = \mathop{\rm argmin}\limits_{v\in V}\left\{
		\dual{\nabla f(y_k),v} + \frac{\gamma_{k}+\mu\alpha_k}{2\alpha_k}\nm{v-w_k}^2
		\right\},\\
		{}&x_{k+1} = \frac{x_k+\alpha_kv_{k+1}}{1+\alpha_k}.
	\end{aligned}
	\right.
\end{equation}
For minimizing $f=h+g$ over $Q$, we modify the above method as follows
\begin{equation}\label{eq:AFBM-1st-argmin}
	\left\{
	\begin{aligned}
		{}&y_k = \frac{x_k+\alpha_kv_k}{1+\alpha_k},\quad
		w_k = {}\frac{\gamma_kv_k+\mu\alpha_ky_k}{\gamma_k+\mu\alpha_k},\\
		{}&v_{k+1} = \mathop{\rm argmin}\limits_{v\in Q}\left\{
		g(v) + \dual{\nabla h(y_k),v} + \frac{\gamma_{k}+\mu\alpha_k}{2\alpha_k}\nm{v-w_k}^2
		\right\},\\
		{}&x_{k+1} = \frac{x_k+\alpha_kv_{k+1}}{1+\alpha_k},
	\end{aligned}
	\right.
\end{equation}
where $x_0,\,v_0\in Q$ and the parameter sequence $\{\gamma_k\}$ comes from the implicit discretization~\eqref{eq:im-gammat-NAG} of the 
equation~\eqref{eq:gammat-NAG}. Clearly, as convex combinations are used, the method~\eqref{eq:AFBM-1st-argmin} preserves the three-term sequence $\{(x_k,y_k,v_k)\}$ in $Q$ and it requires the proximal computation of $g$ over $Q$ only once in each iteration.

We choose $L\alpha_k^2=\gamma_{k}(1+\alpha_k)$ as before and 
rewrite~\eqref{eq:AFBM-1st-argmin} in Algorithm \ref{algo:AFB}, which is called semi-implicit accelerated forward-backward (Semi-AFB for short) method.

\begin{algorithm}[h]
	\caption{Semi-AFB method for solving $\min_{x\in Q} \left [h(x)+g(x) \right ]$}
	\label{algo:AFB}
	\begin{algorithmic}[1] 
		\REQUIRE  $x_0,v_0\in Q,\,\gamma_0>0$ and $L>0$.
		\FOR{$k=0,1,\ldots$}
		\STATE Compute $\alpha_k>0$ such that $L\alpha_k^2= \gamma_k\big(1+\alpha_k\big)$.
		\STATE Update $\displaystyle{\gamma_{k+1} = \frac{\gamma_k+\mu\alpha_k}{1+\alpha_k}}$.
		\STATE Set $\displaystyle{y_k = \frac{x_k+\alpha_kv_k}{1+\alpha_k}}$ and $\displaystyle{w_k = {}\frac{\gamma_kv_k+\mu\alpha_ky_k}{\gamma_k+\mu\alpha_k}}$.
		\STATE Update $\displaystyle{v_{k+1} = {}\mathop{\rm argmin}\limits_{v\in Q}\left\{
			g(v) + \dual{\nabla h(y_k),v} + \frac{\gamma_{k}+\mu\alpha_k}{2\alpha_k}\nm{v-w_k}^2
			\right\}}$.			
		\STATE Update $\displaystyle{x_{k+1} = {}\frac{x_k+\alpha_kv_{k+1}}{1+\alpha_k}}$.	
		\ENDFOR
	\end{algorithmic}
\end{algorithm}

In~\cite{Tseng2008}, Tseng considered problem~\eqref{min-comp} only with convex assumption, i.e., $\mu=0$, and proposed an 
APGM that possesses the rate $O(L/k^2)$. By using the technique of estimate sequence, Nesterov~\cite{Nesterov_2012} presented an accelerated method for solving~\eqref{min-comp} with the assumption that $h$ is $L$-smooth over $Q$ and $g$ is $\mu$-strongly convex with $\mu\geqslant 0$. Both our Algorithm \ref{algo:AFB} and Nesterov's method generate a three-term sequence $\{(x_k,y_k,v_k)\}$ and have the same accelerated rate $O(\min(L/k^2,(1+\sqrt{\mu/L})^{-k}))$; see \cite[Theorem 6]{Nesterov_2012} and our Theorem \ref{thm:conv-AFB}. However, as mentioned in~\cite{Beck2009}, 
the later used an accumulated history of the past iterations to build recursively a sequence of estimate functions, and in each iteration, to update $x_{k+1}$ and $v_{k+1}$, Nesterov's method in~\cite{Nesterov_2012} calls $\prox_{ g}$ over $Q$ twice. 

Below, we shall establish the convergence rate of Algorithm \ref{algo:AFB} via the analysis of a Lyapunov function. It is well known~\cite[Eq (2.9)]{Nesterov_2012} that the first-order optimality condition for $v_{k+1}$ in \eqref{eq:AFBM-1st-argmin} is the variational inequality
\[
\dual{\nabla h(y_k)+\frac{\gamma_{k}+\mu\alpha_k}{\alpha_k}(v_{k+1}-w_k)
	+p_{k+1},x-v_{k+1}}\geqslant 0\quad\forall\,x\in Q,
\]
where $p_{k+1}\in\partial g(v_{k+1})$. Expanding $w_k$, we observe the relation 
\begin{equation}\label{eq:vk1-ineq}
	\begin{split}
		{}&	\gamma_k\inner{v_{k+1}-v_k, v_{k+1} -x} \\
		\leqslant  {}&
		\mu\alpha_k\inner{y_{k} - v_{k+1}, v_{k+1} - x}
		-
		\alpha_k \dual{\nabla h(y_k)+p_{k+1}, v_{k+1} - x},
	\end{split}
\end{equation}
where $x\in Q$ is arbitrary.
\begin{theorem}
	\label{thm:conv-AFB}
	For Algorithm \ref{algo:AFB}, we have 
	\begin{equation}	\label{eq:diff-Lk-AFBM-1st}
		\mathcal L_{k+1}\leqslant 
		\frac{	 \mathcal L_k}{1+\alpha_k }\quad\forall\,k\in\mathbb N,
	\end{equation}
	where $\mathcal L_k=
	f(x_k)-f(x^*) +
	\frac{\gamma_k}{2}
	\nm{v_k-x^*}^2$, and both~\eqref{eq:rate-Lk-ex0} and 
	\eqref{eq:rate-Lk-ex0-L} hold true here.
\end{theorem}
\begin{proof}	
	As before, we calculate the difference
	\[
	\begin{split}
		\mathcal L_{k+1}-\mathcal L_{k}
		={}& f(x_{k+1})-f(x_k)+
		\frac{\alpha_k}{2}(\mu- \gamma_{k+1})
		\nm{v_{k+1}-x^*}^2\\
		{}&	+\gamma_k
		\inner{v_{k+1}-v_k, v_{k+1} -x^*} -
		\frac{\gamma_k}2
		\nm{v_{k+1}-v_k}^2.
	\end{split}
	\]
	Thanks to~\eqref{eq:vk1-ineq}, we have
	\begin{equation}\label{eq:diff-vk}
		\begin{split}
			{}&	\gamma_k\inner{v_{k+1}-v_k, v_{k+1} -x^*} \\
			\leqslant  {}&
			\mu\alpha_k\inner{y_{k} - v_{k+1}, v_{k+1} - x^{*}}
			-
			\alpha_k \dual{\nabla h(y_k)+p_{k+1}, v_{k+1} - x^{*}}.
		\end{split}
	\end{equation}
	where $p_{k+1}	\in\partial g(v_{k+1})$.
	By Lemma \ref{lem:cross}, the first term in~\eqref{eq:diff-vk} is split as follows
	\begin{equation*}
		\begin{split}
			{}&2\mu\alpha_k\inner{y_{k} - v_{k+1}, v_{k+1} - x^{*}} \\
			={}&\mu\alpha_k\left(\left\|y_{k}-x^{*}\right\|^{2}
			-\|y_{k}-v_{k+1}\|^{2}
			-\left\|v_{k+1}-x^{*}\right\|^{2}\right).
		\end{split}
	\end{equation*}
	
	The gradient term in~\eqref{eq:diff-vk} is more subtle. 
	Firstly, by convexity of $g$, we have 
	\[
	\begin{split}
		{}&
		-\alpha_k \dual{p_{k+1}, v_{k+1} - x^{*}}
		\leqslant -\alpha_k\left(g(v_{k+1})-g(x^*)\right)\\
		={}&-\alpha_k\left(g(x_{k+1})-g(x^*)\right)
		-\alpha_k\left(g(v_{k+1})-g(x_{k+1})\right),
	\end{split}
	\]
	and secondly, according to the update for $y_{k}$ 
	(see step 4 in Algorithm \ref{algo:AFB}), we find
	\begin{align*}
		{}&
		-\alpha_k \dual{ \nabla h(y_k), v_{k+1} - x^{*}} \\
		={}&-\alpha_k\dual{\nabla h(y_k),v_{k+1} - v_k} 
		-\alpha_k\dual{  \nabla h(y_k), v_{k} - x^{*}}\\
		={}&-\alpha_k\dual{\nabla h(y_k),v_{k+1} - v_k}
		- \dual{ \nabla h(y_k), y_{k} - x_{k}} 
		- \alpha_k\dual{  \nabla h(y_k), y_{k} - x^{*}}.
	\end{align*}
	As $h$ is $\mu$-strongly convex on $Q$, by the fact $\{(x_k,y_k,v_k)\}\subset Q$, it follows that
	\begin{align*}
		{}&
		- \dual{ \nabla h(y_k), y_{k} - x_{k}} 
		- \alpha_k\dual{  \nabla h(y_k), y_{k} - x^{*}}\\
		\leqslant {}&h(x_k)-h(y_k) -\frac{\mu}{2}\nm{x_k-y_k}^2
		- \alpha_k\left(h(y_k)-h(x^*)\right) 
		-\frac{\mu\alpha_k}{2}\nm{x^*-y_k}^2\\
		={}&
		(1+\alpha_k)\left(h(x_{k+1})-h(y_{k})\right)
		- \alpha_k\left(h(x_{k+1})-h(x^*)\right) 
		-\frac{\mu\alpha_k}{2}\nm{x^*-y_k}^2\\		
		{}&\qquad+h(x_{k})-h(x_{k+1}) -\frac{\mu}{2}\nm{x_k-y_k}^2.
	\end{align*}
	Therefore, collecting all the estimates and dropping surplus negative terms 
	related to $-\nm{x_k-y_k}^2$ and $-\|y_{k}-v_{k+1}\|^{2}$, we get
	\begin{equation}\label{eq:diff-Lk-AFBM-1st-mid}
		\begin{aligned}
			{}&
			\mathcal L_{k+1}-\mathcal L_{k}\\
			\leqslant {}&-\alpha_k\mathcal L_{k+1}	+(1+\alpha_k)\left(h(x_{k+1})-h(y_{k})\right)
			-\alpha_k\dual{\nabla h(y_k),v_{k+1} - v_k}\\
			{}&	-	\frac{\gamma_k}2	\nm{v_{k+1}-v_k}^2
			+g(x_{k+1})-g(x_{k})-\alpha_k\left(g(v_{k+1})-g(x_{k+1})\right).
		\end{aligned}
	\end{equation}
	
	Let us consider the additional terms in~\eqref{eq:diff-Lk-AFBM-1st-mid}.
	In view of~\eqref{eq:Lip-L-equiv}, we have 
	\[
	h(x_{k+1})-h(y_k)\leqslant \dual{\nabla h(y_k),x_{k+1}-y_k} + \frac{L}{2}\nm{x_{k+1}-y_k}^2.
	\]
	Thanks to the extrapolation step for $x_{k+	1}$ (see step 6 in 
	Algorithm \ref{algo:AFB}), we find a crucial relation
	\[
	x_{k+1}-y_k = \frac{\alpha_k}{1+\alpha_k}(v_{k+1}-v_k),
	\]
	which gives that
	\[
	\begin{split}
		{}&(1+\alpha_k)\left(h(x_{k+1})-h(y_{k})\right)
		-\alpha_k\dual{\nabla h(y_k),v_{k+1} - v_k}
		-	\frac{\gamma_k}2	\nm{v_{k+1}-v_k}^2\\
		\leqslant {}&\frac{L\alpha_k^2}{2(1+\alpha_k)}\nm{v_{k+1}-v_k}^2
		-	\frac{\gamma_k}2	\nm{v_{k+1}-v_k}^2=0,
	\end{split}
	\]
	as $L\alpha_k^2=\gamma_k(1+\alpha_k)$. Moreover, since $x_{k+1}$ is a convex combination of $x_k$ and $v_{k+1}$, the estimate follows
	\[
	\begin{split}
		{}&g(x_{k+1})-g(x_{k})-\alpha_k\left(g(v_{k+1})-g(x_{k+1})\right)	\\
		=	{}&(1+\alpha_k)g(x_{k+1})-g(x_{k})-\alpha_kg(v_{k+1})
		\leqslant{}0.
	\end{split}
	\]
	Plugging this and the previous inequality into~\eqref{eq:diff-Lk-AFBM-1st-mid} gives
	\[
	\mathcal L_{k+1}-\mathcal L_{k}\leqslant -\alpha_k\mathcal L_{k+1},
	\]
	which establishes~\eqref{eq:diff-Lk-AFBM-1st}.
	
	By the relation $L\alpha_k^2=\gamma_{k}(1+\alpha_k)$ and the contraction~\eqref{eq:diff-Lk-AFBM-1st}, it is clear that the two estimates~\eqref{eq:rate-Lk-ex0} 
	and~\eqref{eq:rate-Lk-ex0-L} hold true. This completes the proof of this theorem.	
\end{proof}
\vskip0.3cm\noindent{\bf Acknowledgments}
	The authors would like to thank the anonymous reviewers for valuable suggestions and careful comments, which significantly improved the qualify of an early version of the paper. 

\appendix

\section{Spectral Analysis}\label{app:proof}
\medskip\noindent{\bf Proof of Theorem \ref{thm:spec-bd-GS}.}	
Let us start from the scalar case
\[
R = 
\begin{pmatrix}
	-a& \; c\\
	-b& \; -d
\end{pmatrix},
\]
where $a,b,c,d\geqslant 0$ and ${\rm tr}\,R<0<\det R$. 
Set
\[
M = \begin{pmatrix}
	-a& \; 0
	\\-b& \; -d
\end{pmatrix},
\quad 
N = \begin{pmatrix}
	0&  \; c
	\\0& \; 0
\end{pmatrix}.
\]
By direct computation we have 
\begin{equation}\label{eq:Ea}
	\begin{split}
		E(\alpha,R) :={}&(I - \alpha M)^{-1} (I + \alpha N) 
		=\frac{1}{\delta}	\begin{pmatrix}
			1+d\alpha& c\alpha(1+d\alpha)\\
			-b\alpha&
			1+a\alpha-bc\alpha^2
		\end{pmatrix},
	\end{split}
\end{equation}
where $\delta:=(1+a\alpha)(1+d\alpha)$.
Since ${\rm tr}\,R<0$, we see that 
\[
0 < \det E(\alpha,R) = \frac{1}{\delta} =
\frac{1}{1+\snm{{\rm tr}\,R}\alpha+ad\alpha^2}<1.
\]
Note that any eigenvalue $\theta$ of $E(\alpha,R)$ satisfies 
\begin{equation}\label{eq:ta}
	\theta^2-{\rm tr}\,E(\alpha,R)\theta + \det E(\alpha,R)=0.
\end{equation}

We now arrive at the following lemma, which says 
the spectrum of $E(\alpha,R)$ can be transformed 
to the circle $	\snm{\theta} = \sqrt{\det E(\alpha,R)}< 1$, with proper $\alpha$.
\begin{lemma}\label{lem:eig-R}
	Assume
	\[
	R = 
	\begin{pmatrix}
		-a& \; c\\
		-b& \; -d
	\end{pmatrix},
	\]
	with $a,b,c,d\geqslant 0$ such that ${\rm tr}\, R <{} 0<\det R $. 
	Let $E(\alpha,R)$ be defined by~\eqref{eq:Ea}.
	If $\alpha>0$ satisfies
	\begin{equation}\label{eq:cond-a-spec}
		|{\rm tr}\, R| - 2\sqrt{\det R} \leqslant bc \, \alpha  \leqslant  |{\rm tr}\, R| + 2\sqrt{\det R},
	\end{equation}
	then we have 
	\[
	\rho(E(\alpha,R)) =  \frac{1}{\sqrt{1+\snm{{\rm tr}\,R}\alpha+ad\alpha^2}}<1.
	\]	
\end{lemma}
\begin{proof}
	If $\Delta=\snm{{\rm tr}\,E(\alpha,R)}^2-4\det E(\alpha,R)\leqslant 0$, then any solution to~\eqref{eq:ta}  satisfies that $\snm{\theta} = \sqrt{\det E(\alpha,R)}$ and the conclusion follows.
	By direct calculation, $\Delta\leqslant 0$ is equivalent to
	\[
	\sqrt{\delta} - 1 \leqslant\alpha \sqrt{\det R}\leqslant\sqrt{\delta} + 1.
	\]
	Square the inequality $\alpha \sqrt{\det R} - 1 \leqslant\sqrt{\delta}$ and cancel one $\alpha$ to get the upper bound in ~\eqref{eq:cond-a-spec}. The lower bound can be proved similarly. 
	 \end{proof}

We now in the position of establishing Theorem \ref{thm:spec-bd-GS}. 		We first consider $ G= G_{_{\rm HB}}$, for which we have
\[
E(\alpha,G)
=
\frac{1}{1+2\alpha}
\begin{pmatrix}
	(1+2\alpha)I&\quad \alpha(1+2\alpha)I\\
	-\alpha A /\mu&\quad  I-A\alpha^2/\mu
\end{pmatrix}.
\]
It is clear that $\theta \in\sigma (E(\alpha,G))\Leftrightarrow \theta  \in\sigma(E(\alpha,R(\lambda)))$, where $E(\alpha,R(\lambda))$ is defined by~\eqref{eq:Ea} with 
\[
R(\lambda) = 
\begin{pmatrix}
	0&\; 1\\
	-\lambda /\mu&\; -2
\end{pmatrix},\quad \lambda \in\sigma(A).
\]												
As $|{\rm tr}\, R(\lambda)| \leqslant 2\sqrt{\det R(\lambda)}$, by Lemma \ref{lem:eig-R}, if 
\begin{equation}\label{eq:cond-a}
	0<\alpha\leqslant  2/\sqrt{\kappa(A)},
\end{equation}
then we can obtain
\[
\rho(	E(\alpha,G))=\max_{\lambda \in \sigma(A)}\rho(E(\alpha,R(\lambda)))=
\frac{1}{\sqrt{1+2\alpha}}.
\]
Similarly, for $G = G_{_{\rm NAG}}$ with condition~\eqref{eq:cond-a}, we can establish 
\[
\rho(	E(\alpha,G))=\max_{\lambda \in \sigma(A)}\rho(E(\alpha,R(\lambda)))
=	\frac{1}{\sqrt{1+2\alpha+\alpha^2}}
\leqslant \frac{1}{\sqrt{1+2\alpha}}.
\]
Consequently, for both two cases, taking $\alpha = 2/\sqrt{\kappa(A)}$ yields the spectrum bound
\[
\rho(	E(\alpha,G))
\leqslant \frac{1}{\sqrt{1+4/\sqrt{\kappa(A)}}}
\leqslant\frac{1}{1+1/\sqrt{\kappa(A)}}.
\] 	
This concludes the proof of Theorem \ref{thm:spec-bd-GS}.
\hfill\ensuremath{\square}

\medskip\noindent{\bf Proof of Theorem \ref{thm:spec-mu-0}.}
Observe that $\widetilde E_{k}$ is similar with 
\[
\begin{pmatrix}
	I & \; O \\
	O & \; \gamma_{k} I
\end{pmatrix}^{-1}\begin{pmatrix}
	I & \; O \\
	O & \; \gamma_{k+1} I
\end{pmatrix}
E(\alpha_{k}, G(\gamma_{k+1}))=
\frac{H_{k}}{1+\alpha_{k}},
\]
where 
\[
H_{k}=\begin{pmatrix}
	I & \; \alpha_{k}I \\
	-A\alpha_{k}/\gamma_{k}& \quad
	I-A\alpha_{k}^2/\gamma_{k}
\end{pmatrix}.
\]
To prove~\eqref{eq:rho}, it is sufficient to verify $\rho \left(H_{k}\right)=1$.

Given any eigenvalue $\theta\in\sigma(H_{k})$, it solves 
\[
\theta^2+(\lambda\alpha_{k}^2/\gamma_{k}-2)\theta+1=0,
\]
with some $\lambda\in\sigma (A)\subset
[0,L]$. By~\eqref{eq:im-gk}, $\{\gamma_k\}$ is decreasing and thus $\gamma_{k}\leqslant \gamma_0=L$. According to our choice $L\alpha_k^2=\gamma_k(1+\alpha_k)$, we have $0<\alpha_k\leqslant 2$ and moreover $0<\lambda\alpha_k^2/\gamma_k\leqslant L\alpha_k^2/\gamma_k = 1+\alpha_k\leqslant 3$. This implies 
$\Delta = (\lambda\alpha_{k}^2/\gamma_{k}-2)^2-4\leqslant 0$ for all $\lambda\in\sigma (A)$. Therefore, we conclude that $\snm{\theta} = 1$ for all $\theta\in\sigma(H_{k})$, which proves $\rho \left(H_{k}\right)=1$ and thus establishes~\eqref{eq:rho}.

Thanks to Lemma \ref{lem:gk-ak-im}, there holds
\[
\frac{1}{(k+1)^2}
\leqslant 
\frac{\gamma_{k}}{\gamma_0}=
\prod_{i=0}^{k-1}\frac{1}{1+\alpha_i}\leqslant\frac{4}{(k+2)^2}.
\]
This proves~\eqref{eq:rho-prod} and completes the 
proof of Theorem \ref{thm:spec-mu-0}.
\hfill\ensuremath{\square}

\section{Decay Rates}
\begin{lemma}
	\label{lem:gk-ak}
	Let $\gamma_0>0$ and $ \mu\geqslant0$ be given and assume there is a real positive sequence $\{L_k\}$ such that $ L_k\geqslant   \mu$. Define $\{(\alpha_k,\gamma_k)\}$ by that
	\begin{equation}	\label{eq:gk-ak}
		\left\{
		\begin{split}
			L_k\alpha_k^2 = {}&\gamma_{k+1},\quad \alpha_k>0,\\
			\gamma_{k+1} = {}&(1-\alpha_k)\gamma_k+\mu\alpha_k.
		\end{split}
		\right.
	\end{equation}
	Then we have $\gamma_k>0,0<\alpha_k\leqslant 1$ and $\alpha_k
	\geqslant \sqrt{\min\{\gamma_1,\mu\}/L}$, where $L: = \sup_{k\in\mathbb N}L_k$. Moreover, for all $k\geqslant 1$,
	\begin{equation}\label{eq:upbd}
		\prod_{i=0}^{k-1}(1-\alpha_i)\leqslant
		\min\left\{
		4\left(2+\sum_{i=0}^{k-1}\sqrt{\frac{\gamma_0}{L_i}}\right)^{-2}
		,\,
		\left(1-\sqrt{\frac{\min\{\gamma_1,\mu\}}{L}}\right)^{k}
		\right\},
	\end{equation}
	and if $\mu=0$, then we have the lower bound
	\begin{equation}\label{eq:lowbd}
		\prod_{i=0}^{k-1}(1-\alpha_i)\geqslant
		\left(1+\sum_{i=0}^{k-1}\sqrt{\frac{\gamma_0}{L_i}}\right)^{-2}.
	\end{equation}
\end{lemma}
\begin{proof}
	Let us first check that $0<\alpha_k\leqslant 1$ and $\gamma_k>0$.
	Since $\gamma_0>0$, by~\eqref{eq:gk-ak} we have 
	\[
	L_0\alpha_0^2 = \gamma_1=(1-\alpha_0)\gamma_0+\mu\alpha_0,
	\]
	from which we claim that $0< \alpha_0\leqslant 1$. Thus by the second step in~\eqref{eq:gk-ak} we have $\gamma_1>0$. A sequential argument implies that $0< \alpha_k\leqslant 1$ and $\gamma_k>0 $ for all $k\geqslant 0$. 
	
	It is not hard to find the fact: if $\gamma_0>\mu$, then $\mu<\gamma_{k+1}<\gamma_k$ and if $\gamma_0<\mu$, then $\gamma_{k}<\gamma_{k+1}<\mu$. 
	Particularly, if $\gamma_0=\mu$, then $\gamma_k=\mu$. 
	Based on this observation and the fact $L_k\leqslant L$, we 
	conclude that $\alpha_k\geqslant \sqrt{\min\{\gamma_1,\mu\}/L}$ and thus
	\[
	\prod_{i=0}^{k-1}(1-\alpha_i)\leqslant
	\left(1-\sqrt{\frac{\min\{\gamma_1,\mu\}}{L}}\right)^{k}.
	\]
	
	Next, let us prove the estimate
	\begin{equation}\label{eq:upbd-2}
		\rho_k
		\leqslant
		4\left(2+\sum_{i=0}^{k-1}\sqrt{\frac{\gamma_0}{L_i}}\right)^{-2},
	\end{equation}
	where  $\rho_k$ is defined by~\eqref{eq:rhok-1-a}. 
	We start from the trivial equality
	\begin{equation}\label{eq:1/gk}
		\frac{1}{\sqrt{\rho_{k+1}}} - \frac{1}{\sqrt{\rho_{k}}}
		=\frac{\sqrt{\rho_{k}}-\sqrt{\rho_{k+1}}}
		{\sqrt{\rho_{k}\rho_{k+1}}}=
		\frac{1-\sqrt{1-\alpha_k}}{\sqrt{\rho_{k+1}}}
		=\frac{\alpha_k}{\sqrt{\rho_{k+1}}(1+\sqrt{1-\alpha_k})},
	\end{equation}
	where we used the relation $\rho_{k+1}=\rho_{k}(1-\alpha_k)$.
	By~\eqref{eq:gk-ak}, for any $i\geqslant 0$, it holds that
	\begin{equation}\label{eq:rhok-gi}
		\gamma_{i+1}=(1-\alpha_i)\gamma_{i}+\mu\alpha_i
		\geqslant (1-\alpha_i)\gamma_i,
	\end{equation}
	and multiplying the above inequality from $i=0$ to $i=k-1$ gives $\rho_{k}\leqslant \gamma_{k}/\gamma_0$.
	Plugging this into~\eqref{eq:1/gk} and using the relation $L_k\alpha_k^2=\gamma_{k+1}$ and the fact $0<\alpha_k\leqslant 1$ imply
	\[
	\frac{1}{\sqrt{\rho_{k+1}}} - \frac{1}{\sqrt{\rho_{k}}}
	\geqslant\frac{\sqrt{\gamma_0}\alpha_k}
	{\sqrt{\gamma_{k+1}}(1+\sqrt{1-\alpha_k})}
	\geqslant\frac{\sqrt{\gamma_0}}{2\sqrt{L_k}},
	\]
	which further indicates that
	\[
	\frac{1}{\sqrt{\rho_{k}}} - 
	\frac{1}{\sqrt{\rho_{0}}} 
	\geqslant \sum_{i=0}^{k-1}\frac{\sqrt{\gamma_0}}{2\sqrt{L_i}}.
	\]
	Therefore, a simple calculation proves~\eqref{eq:upbd-2} and 
	concludes the proof of this lemma.
	
	For $\mu=0$, we have the relation $\rho_{k}=\gamma_{k}/\gamma_0$, and proceeding as the above derivation, it is not hard to establish the lower bound~\eqref{eq:lowbd}. This
	concludes the proof of this lemma.
	
\end{proof}

Similarly, we can establish the following result, the proof of which is omitted for simplicity.
\begin{lemma}
	\label{lem:gk-ak-im}
	Let $\gamma_0>0$ and $ \mu\geqslant0$ be given and assume there is a real positive sequence $\{L_k\}$ such that $ L_k\geqslant   \mu$. Define $\{(\alpha_k,\gamma_k)\}$ by that
	\[
	\left\{
	\begin{split}
		\gamma_{k+1} = {}&\gamma_k+\alpha_k(\mu-\gamma_{k+1}),\\
		L_k\alpha_k^2 = {}&\gamma_{k}(1+\alpha_k),\, \alpha_k>0.
	\end{split}
	\right.
	\]
	Then we have $\gamma_k>0$ and 
	$\alpha_k\geqslant \sqrt{\min\{\gamma_0,\mu\}/L}$, where $L: = \sup_{k\in\mathbb N}L_k$. Moreover, for all $k\geqslant 1$,
	\[
	\prod_{i=0}^{k-1}\frac{1}{1+\alpha_i}\leqslant
	\min\left\{
	4\left(2+\sum_{i=0}^{k-1}\sqrt{\frac{\gamma_0}{L_i}}\right)^{-2}
	,\,
	\left(1+\sqrt{\frac{\min\{\gamma_0,\mu\}}{L}}\right)^{-k}
	\right\},
	\]
	and if $\mu=0$, then we have the lower bound
	\[
	\prod_{i=0}^{k-1}\frac{1}{1+\alpha_i}\geqslant
	\left(1+\sum_{i=0}^{k-1}\sqrt{\frac{\gamma_0}{L_i}}\right)^{-2}.
	\]
\end{lemma}

\end{document}